\documentclass[11pt,letterpaper,reqno,oneside]{amsart}
\pdfoutput=1

\usepackage[utf8]{inputenc} 
\usepackage[normalem]{ulem}
\usepackage{amsmath, amssymb,bm, cases, mathtools, thmtools}
\usepackage{verbatim}
\usepackage{graphicx}\graphicspath{{figures/}}
\usepackage{multicol}
\usepackage{tabularx}
\usepackage[usenames,dvipsnames]{xcolor}
\usepackage{mathrsfs} 
\usepackage{url}
\usepackage{comment}
\usepackage{microtype}

\usepackage[%
    minnames=1,maxnames=99,maxcitenames=4,
    style=alphabetic,
    doi=false,url=false,
    firstinits=true,hyperref,natbib,backend=bibtex]{biblatex}
\renewbibmacro{in:}{%
  \ifentrytype{article}{}{\printtext{\bibstring{in}\intitlepunct}}}
\AtEveryBibitem{\clearfield{issn}}
\AtEveryBibitem{\clearfield{isbn}}
\AtEveryCitekey{\clearfield{issn}}
\AtEveryCitekey{\clearfield{isbn}}
\AtEveryBibitem{\ifentrytype{book}{\clearfield{pages}}{}}
\DeclareFieldFormat[book,article,unpublished,misc]{title}{\textit{#1}}
\DeclareFieldFormat[article]{journaltitle}{{#1}}

\bibliography{conditioning}

\usepackage[colorlinks,citecolor=darkblue,urlcolor=darkblue,linkcolor=darkgreen]{hyperref}
\usepackage{hypernat}

\usepackage{datetime}

\usepackage[capitalize]{cleveref}  
\crefname{lem}{Lemma}{Lemmas}
\crefname{cor}{Corollary}{Corollaries}
\crefname{thm}{Theorem}{Theorems}
\crefname{assumption}{Assumption}{Assumptions}
\crefname{equation}{Equation}{Equations}

\usepackage[left=1.5in,right=1.5in,bottom=1.5in,top=1.5in]{geometry}
\newcommand{\Directed}{\mathscr{D}}

\newcommand{\multito}{\rightrightarrows}

\definecolor{LightGray}{rgb}{.8,.8,.8}

\definecolor{darkred}{rgb}{0.5,0,0}
\definecolor{darkgreen}{rgb}{0, 0.3,0}
\definecolor{darkblue}{rgb}{0,0,0.6}

\declaretheorem[style=plain,numberwithin=section,name=Theorem]{theorem}
\declaretheorem[style=plain,sibling=theorem,name=Lemma]{lemma}
\declaretheorem[style=plain,sibling=theorem,name=Proposition]{proposition}
\declaretheorem[style=plain,sibling=theorem,name=Corollary]{corollary}
\declaretheorem[style=plain,sibling=theorem,name=Claim]{claim}

\declaretheorem[style=definition,sibling=theorem,name=Definition]{definition}

\declaretheorem[style=remark,qed=$\triangleleft$,sibling=theorem,name=Remark]{remark}

\declaretheoremstyle[
    spaceabove=0pt, 
    spacebelow=6pt, 
    headfont=\normalfont\itshape, 
    bodyfont = \normalfont,
    postheadspace=1em, 
    qed=$\triangleleft$, 
    headpunct={.}]{myproofstyle} 
\declaretheorem[name={Proof of Claim}, style=myproofstyle, unnumbered]{claimproof}

\numberwithin{equation}{section}
\numberwithin{theorem}{section}

\newcommand{\Borel}[1]{\mathcal{B}_{#1}}

\makeatletter
\def\imod#1{\allowbreak\mkern10mu({\operator@font mod}\,\,#1)}
\makeatother

\newcommand{\II}{\mathbb{I}}

\newcommand{\Measures}{\mathcal{M}}
\newcommand{\ProbMeasures}{\Measures_1}

\newcommand{\WW}{\mathrm{W}}

\newcommand{\eqW}{\equiv_{\WW}}
\newcommand{\leW}{\le_{\WW}}

\newcommand{\sWW}{\mathrm{sW}}

\newcommand{\eqsW}{\equiv_{\sWW}}
\newcommand{\lesW}{\le_{\sWW}}

\newcommand{\Baire}{\ensuremath{\Nats^\Nats}}

\newcommand{\pairing}[1]{\langle #1\rangle}

\newcommand{\defn}[1]{{\bf{#1}}}

\DeclareMathOperator{\supp}{supp}

\newcommand{\M}{{\mathcal{M}}}

\renewcommand{\L}{\\L}

\newcommand{\Naturals}{\mathbb{N}}
\newcommand{\Nats}{\Naturals}

\newcommand{\Sierpinski}{\mathbb{S}}
\newcommand{\Sier}{\Sierpinski}

\DeclareMathOperator{\EC}{EC}
\DeclareMathOperator{\DD}{D}
\newcommand{\Dstar}[1]{\DD^\star_{#1}}

\DeclareMathOperator{\id}{id}
\DeclareMathOperator{\Lim}{Lim}

\newcommand{\ACD}{\mathcal{C}}

\renewcommand{\Pr}{\mathbf{P}}
\newcommand{\defas}{:=}

\newcommand{\dee}{\mathrm{d}}

\newcommand{\st}{\,:\,}

\def\<{\langle}
\def\>{\rangle}
\def\^{{}^{\wedge}}

\def\[#1\]{\begin{align}#1\end{align}}

\newcommand{\pars}{\,\cdot\,}

\newcommand{\w}{\ensuremath{\omega}}
\newcommand{\Reals}{\ensuremath{\mathbb{R}}}
\newcommand{\Rationals}{\ensuremath{\mathbb{Q}}}

\newcommand{\Cantor}{{\ensuremath{2^\w}}}

\newcommand{\PF}[2]{{#1}[#2]}

\newcommand{\icmu}[1]{#1^{\circ \mu}}
\newcommand{\icmus}[2]{#1^{\circ #2}}

\newcommand{\HIDE}[1]{}

\newcommand{\DUP}[1]{}
\newcommand{\metric}{{d}}
\newcommand{\pmetric}{\metric_{\ProbMeasures}}
\newcommand{\tmetric}{\metric_T}

\newcommand{\AD}[1]{\mathscr{O}_{#1}}
\newcommand{\CB}[1]{\mathscr{B}_{#1}}
\newcommand{\Balls}{B}

\newcommand{\ContFuncs}{\mathscr{C}}
\newcommand{\Opens}{\mathscr{O}}

\newcommand{\ZO}{\mathbb B}

\newcommand{\LReals}{\Reals_{\prec}}
\newcommand{\UReals}{\Reals_{\succ}}

\begin{document}

\title[On computability and disintegration]
{On computability and disintegration}

\author[Ackerman]{Nathanael~L. Ackerman}
\address{
Department of Mathematics \\
Harvard University \\
One Oxford Street, Cambridge, MA 02138, USA
}
\email{nate@math.harvard.edu}
\thanks{}

\author[Freer]{Cameron~E. Freer}
\address{
Department of Brain and Cognitive Sciences\\
Massachusetts Institute of Technology\\
77 Massachusetts Ave.\\
Cambridge, MA 02139\\
USA\\
and
Gamalon Labs\\
One Broadway\\
Cambridge, MA 02142\\
USA
}
\email{freer@mit.edu}

\thanks{}

\author[Roy]{Daniel~M. Roy}
\address{
Department of Statistical Sciences\\
University of Toronto\\
100 St.\ George St.\\
Toronto, ON M5S 3G3\\
Canada
}
\email{droy@utstat.toronto.edu}
\thanks{}



\begin{abstract}
We show that the disintegration operator on 
a complete separable metric space
along a
projection map, restricted to measures for which there is a unique
continuous disintegration,
is 
strongly 
Weihrauch equivalent to the limit operator $\Lim$.
When a measure does not have a unique continuous disintegration, we may still obtain a disintegration 
when some
basis of continuity sets  
has
the Vitali covering property with respect to the measure;
the disintegration, however, may depend on the choice of sets.
We show that, when the basis is computable, the resulting disintegration is 
strongly
Weihrauch reducible to $\Lim$, and further exhibit a single distribution
realizing this upper bound.
\end{abstract}


{\let\thefootnote\relax\footnotetext{
%
}}


\maketitle

\setcounter{page}{1}
\thispagestyle{empty}

\begin{minipage}{.040\linewidth}
\ 
\end{minipage}%
\begin{minipage}{.82\linewidth}
\begin{scriptsize}
\renewcommand\contentsname{\!\!\!\!}
\setcounter{tocdepth}{3}
\tableofcontents
\end{scriptsize}
\end{minipage}


\newpage

\section{Introduction}
\label{IntroSec}

Conditioning is a basic tool in probability theory 
and the core operation in Bayesian statistics.
Given a pair of random variables $\theta$ and $D$, 
representing, e.g., an unobserved parameter of interest and some data collected in order to estimate the parameter, 
the conditional distribution of $\theta$ given $D=x$ can be understood as the distribution $\mu^x$
that assigns to every measurable set $A$ the probability 
\[\label{basicdef}
	\mu^x(A) = \frac {\Pr \{ \theta \in A \, \textrm{~and~}\, D = x \} } { \Pr \{ D = x \} },
\]
provided  that
\[\label{positive} 
\Pr \{ D = x \} > 0.
\]
When $\mu^x$ is defined for all $x$ in the range of $D$,
the map $x \mapsto \mu^x$ is called a \emph{disintegration} of the distribution of $\theta$ with respect to $D$.

When property \eqref{positive} does not hold, \cref{basicdef} is meaningless. 
In modern probability theory, rather than defining the conditional probability $\mu^x(A)$ for a particular $x$, 
the map $x \mapsto \mu^x(A)$
is defined all at once as a function $f$ satisfying
\[
	\Pr \{ \theta \in A \,\textrm{~and~}\, D \in B \} = \int_B f(x)\, \Pr_D (\dee x),
\]
for all measurable subsets $B$, where $\Pr_D \defas \Pr \{D \in \cdot\, \}$ is the distribution of $D$.  
The existence of such an $f$ is guaranteed by the Radon--Nikodym theorem. 
The resulting disintegrations are, however, defined only up to a null set, and so
their evaluation at points---namely, at the observed data values, as is statistical practice---has typically
relied upon additional (often unstated) hypotheses.  

One such hypothesis is the continuity of \emph{some} disintegration, which ensures that 
any continuous disintegration
is
canonically defined everywhere in the support of the distribution of the
conditioning variables. It is interesting to consider the computability
of the conditioning operator in this context.  

In \cite{AFR10} and \cite{DBLP:conf/lics/AckermanFR11}, we showed that
there are computable random variables whose disintegrations are continuous but not computable on any measure one set (indeed, on any set of sufficiently large measure).
Here we strengthen these
results and make them uniform by providing precise bounds on how noncomputable 
disintegration can be.

In the present paper we make use of 
constructive definitions of
disintegrations at a point $x$ in terms of limits of the form
\[
	\lim_{n \to \infty} \frac {\Pr \{ \theta \in A \,\textrm{~and~}\, D \in B_n \} } { \Pr \{ D \in B_n \} }.
\]
for some sequence of sets $B_n$ ``converging'' to $x$ in an appropriate sense.
We make repeated use of notions developed by Tjur \cite{MR0123456,MR595868}. 
Namely, Tjur
introduced a property that may hold of a distribution \emph{at a point}
which implies
that all sensible choices of $B_n$ lead to the same solution. 
Furthermore, under certain
regularity conditions, if a distribution satisfies Tjur's property at every point, and so can be conditioned at every point, then 
the resulting disintegration is the unique continuous disintegration---the ideal case for statistical purposes.
We show that the disintegration 
operator on the collection of distributions (on a fixed space) for which
the Tjur property holds everywhere is strongly 
Weihrauch equivalent to the $\Lim$ operator.

Tjur's property is rather
restrictive, and so we also explore some other attempts to give explicit formulas for disintegrations in terms of limits.
Other work in this direction includes Kolmogorov's original axiomatization~\cite{MR0362415}
and the work of Pfanzagl~\cite{MR548898}
and of Fraser and Naderi~\cite{MR1462405}.
The latter two papers approach conditioning as the differentiation of a set function.  Pfanzagl's definition extends Kolmogorov's own observation that limits in the case of the real line admit conditional probabilities.  In general, the resulting disintegrations are only defined up to null sets and the resulting conditional probability will depend on the chosen sequence of sets $B_n$. Fraser and Naderi give conditions on a collection of sets that imply that there is no such dependence, and we use their notions in 
Section~\ref{specific}
to show that an everywhere-defined disintegration is strongly Weihrauch reducible to $\Lim$, and moreover, that this upper bound is achieved by a particular disintegration.

Recently there has been a great deal of progress towards the general program of determining the complexity of various operations in analysis by placing the relevant operators within the hierarchy of Weihrauch degrees (see, e.g., \cite{MR2760117} and \cite{MR2791341}). 
The Lim operator plays a central
role among the Weihrauch degrees, analogous to that of the halting set among Turing degrees. (For background and various equivalences, see
\cite[\S3]{Brattka2012623}.)
Viewed within this program, the present paper extends certain consequences of results about Radon--Nikodym derivatives. Namely, a
result by Hoyrup, Rojas, and Weihrauch \cite{HRWCompJ} (see also \cite{HR11}) can be shown to imply that the map taking a
distribution and a real $\varepsilon > 0$ to a continuous disintegration on
some $(1-\varepsilon)$-measure set is Weihrauch reducible to the limit
operator $\Lim$, a result which is strengthened by our work.

\subsection{Outline of the paper}
In Section~\ref{Tjursection}, we describe 
necessary and sufficient conditions for the existence of a continuous disintegration
due to Tjur \cite{MR0123456}, as well as weaker sufficient conditions for the existence of a disintegration
due to Fraser and Naderi \cite{MR1462405}.
Then in Section~\ref{Wred}, we 
recall the relevant definitions about Weihrauch degrees and represented spaces. We also define and describe some basic properties of the operators $\EC$ and $\Lim$, which are key in our arguments (see Definitions~\ref{EC definition} and \ref{Lim definition} for their definitions).

Next, in Section~\ref{continuous}, we
use Tjur's characterization 
to show that in a natural setting,
the conditioning operator, which maps a distribution to a continuous
disintegration, is strongly Weihrauch equivalent to $\Lim$.
In particular, in \S\ref{lower-uniform},
we use a relativization of a simplified variant of the main construction of
\cite{AFR10}
to obtain a lower bound in the Weihrauch degrees of the conditioning operator.

In Section~\ref{specific}, we also consider those measures for which the disintegration is not
continuous, and show that under certain conditions, when the disintegration
exists everywhere, the disintegration is strongly Weihrauch reducible to $\Lim$.
We also exhibit
a specific disintegration that realizes this upper bound.

\section{Disintegration}
\label{Tjursection}

In this section we define the abstract notion of disintegration rigorously.
In doing so, we highlight the fact that disintegrations are, in general, not well-defined at points. 
We also discuss various conditions that allow us to 
provide a canonical definition at points.
Note that these definitions are in terms of probability measures rather than in terms of random variables, 
as was the case in Section~\ref{IntroSec}.

For a topological space $X$, write $\Borel{X}$ to denote the Borel $\sigma$-algebra on $X$. When we speak of the \emph{measurable space} $X$ without mentioning the $\sigma$-algebra, the Borel $\sigma$-algebra is intended.

\begin{definition}
Let $S$ be a topological space and let $\ProbMeasures(S)$ denote the collection of Borel probability measures on $S$. 
Then the \defn{weak topology} on $\ProbMeasures(S)$ is that generated by sets of the form 
\[
\bigl\{\mu \in \ProbMeasures(S) \st \bigl | \textstyle \int_S \varphi \,\dee \mu - c \bigr | < \varepsilon\bigr\} ,
\]
where $c, \varepsilon\in \Reals$, $\varepsilon > 0$,
and $\varphi:S \to \Reals$ is bounded and continuous.
\end{definition}

Let $S$ and $T$ be measurable spaces.
Given a measurable map $g\colon S \to T$ and 
probability measure $\mu$ on $S$, define $\PF g \mu$ to be the \defn{pushforward} measure on $T$ given by 
\[
(\PF g\mu)(B) \defas (\mu \circ g^{-1})(B) = \mu\bigl(g^{-1}(B)\bigr)
\]
for all $B \in \Borel{T}$.

\begin{definition}[Disintegration along a map]
Let $S$ and $T$ be measurable spaces, 
let $\mu$ be a probability measure on $S$,
and let $g$ be a measurable function from $S \to T$.
Then a \defn{disintegration of $\mu$ along $g$} is a 
measurable map $\kappa\colon T\to \ProbMeasures(S)$ such that
	\[\label{disinteq}
		\mu\bigl(g^{-1}(B) \cap A\bigr) = \int_B \kappa(x)(A) \ \PF g \mu (\dee x),
	\]
for all $A \in \Borel{S}$
and $B \in \Borel{T}$.
\end{definition}
We will often consider the case where $S$ is a product of two spaces and $g$ is the projection from $S$ onto one of these spaces.

We have chosen to define a disintegration to be a measurable map $\kappa\colon T\to \ProbMeasures(S)$.  As is common in probability theory,
one might instead have chosen to work with the ``uncurried'' \emph{probability kernel} $\kappa^* \colon (T \times \Borel{S}) \to [0,1]$ defined by 
$\kappa^*(t, A) \defas \kappa(t)(A)$ and assumed to be a probability measure for every fixed first parameter and a measurable function for every fixed second parameter.
For more details on the notion of disintegration, see \cite[Ch.~6]{MR1876169}.

Any two maps $\kappa, \kappa'$ satisfying the definition of a disintegration agree on a $\PF g \mu$-measure one set,
and so we call every such map a \emph{version}.
We can speak of \emph{the} disintegration, which is then 
only determined up to a $\mu$-measure zero set. 
However, by adding a continuity requirement we are able to pin down a version uniquely. When we speak about the continuity of disintegrations,
the topology on the space of measures will be the weak topology.

\begin{lemma}\label{contdistareunique}
Let $\mu$ be a distribution on $S$, let $g \colon S \to T$ be measurable, let $\nu = \PF g \mu$, 
and let $\kappa$ and $\kappa'$ be disintegrations of $\mu$ along $g$.
If $b \in T$ is a point of continuity of both $\kappa$ and $\kappa'$, and $b \in \supp(\nu)$,
then $\kappa(b) = \kappa'(b)$.
In particular, if $\kappa$ and $\kappa'$ are continuous everywhere, they agree on $\supp(\nu)$.
\end{lemma}
\begin{proof}
Assume, towards
a contradiction, 
	that
	$b \in \supp(\nu)$ is a point of continuity of both $\kappa$ and $\kappa'$ such that $\kappa(b) \neq \kappa'(b)$. Then
there exists some 
bounded continuous function $f$ such that 
the function $h$ defined by
\[
		h(x) \defas
	\left | \int f(y) \  \kappa(x)(\dee y) -
	\int f(y) \  \kappa'(x)(\dee y)
	\right |
\]
is such that $h(b) > 0$.
Because $\kappa$ and $\kappa'$ are both disintegrations of $\mu$ along $g$,
we have
$h(x) = 0$ for $\nu$-almost all $x$. But
$h$ is continuous at $b$ 
because $\kappa$ and $\kappa'$ are presumed to be continuous at $b$,
and so there is an open neighborhood $N$ of $b$ such that $0 \not\in h(N)$. Finally, as $b\in\supp(\nu)$, we have
$\nu(N) > 0$, a contradiction.
\end{proof}

In particular, given a measure $\mu$ on $S$, if $\supp(\PF g \mu) = T$,
there is at most one disintegration of $\mu$ along $g$ that is continuous on all of $T$.

The question of the existence of a disintegration has a long history.  
One of the most important cases is when $S$ is a Borel space,
	i.e., there exists a measurable bijection, with a measurable inverse, to some Borel subset of the unit interval.
The Disintegration Theorem
implies that 
any measure on $S$ has a disintegration along $g \colon S \to T$
(see, e.g., \cite[Thms.~6.3 and 6.4]{MR1876169}).
Note, however, that the resulting disintegration is only defined up 
to a null set.

For countable discrete spaces, 
the notion of disintegration can be given a concrete definition in terms of the elementary notion of conditioning on positive-measure sets.

\begin{definition}[Conditioning on positive-measure sets]
Given a probability measure $\nu$ on a topological space $S$, and a Borel
set $A$ of $S$ such that $\nu(A)>0$, we write
\[
\nu^A(\pars) \defas \frac{\nu(\pars \cap A)}{\nu(A)}
\]
for \defn{the distribution $\nu$ conditioned on (the set) $A$}.
\end{definition}

A key
relationship between disintegration and conditioning on 
positive-measure subsets of a
countable discrete space $T$ is that the measurable function
$\kappa \colon T \to \ProbMeasures(S)$ defined 
by
\[
\kappa(t) = \mu^{g^{-1}(t)}(\pars),
\]
for every $t$ in the support of $\PF g \mu$, (in particular, defined $\PF g \mu$-a.e.)
is a version of the disintegration of $\mu$ along $g$.  
 
In general, there is no explicit formula like the one for countable discrete spaces.  Special cases, e.g., assuming the existence of joint or conditional densities, admit so-called Bayes' rules, but these hypotheses often exclude infinite-dimensional parameter spaces, which are typical of nonparametric Bayesian statistics.

One might hope that the conditional distribution at a point $x$ would be well-approximated by the conditional distribution given a small positive-measure set ``converging'' to $x$.   
In general, for various natural notions of convergence, this 
need
not hold, but Tjur \cite[\S 9.7]{MR595868} gave conditions for a given probability measure on a point $x$ implying that all ``reasonable'' ways of approximating the conditional distribution by positive-measure sets converging to a point $x$ yield the same conditional distribution in the limit. 
For a fixed measure, when every point possesses this property --- which we call the \emph{Tjur property} --- it follows that there is a unique continuous disintegration.
Indeed, under some mild additional conditions,
a measure satisfies Tjur's property at all points exactly when there exists a unique continuous disintegration.

Tjur's conditions are therefore necessarily rather restrictive. Fraser and Naderi \cite{MR1462405} define a less restrictive notion based on the differentiation of set functions.  Here, the 
Vitali covering property of a class of subsets with respect to the measure of interest 
allows us to explicitly construct a disintegration pointwise. (See \cref{Vitali covering property} for a definition of the Vitali covering property.)

In order to avoid pathologies, all measures in this paper will be Radon. Recall that when $S$ is a Hausdorff space and $\mu \in \ProbMeasures(S)$, $\mu$ is a \defn{Radon measure} when it is inner regular, i.e., whenever $A$ is a Borel set on $S$, then 
\[
\mu(A) = \sup \{\mu(K) \st K \subseteq A\text{ and }K\text{ is compact}\}.
\]
It is a standard result that any Borel probability measure on a separable metric space is a Radon measure.

\subsection{The Tjur property}
\label{Tjurpoints}
The following property, which we have named after Tjur, 
is equivalent to
a property described first by Tjur in an 
unpublished preprint~\cite{MR0123456} and later in a monograph
\cite[\S9.7]{MR595868}.

\begin{definition}[Tjur Property]
\label{TjurProperty}
Let $S$ and $T$ be completely regular Hausdorff spaces, 
let $\mu$ be a Radon probability measure on $S$, 
and let $g\colon S\to T$ be a measurable function.
Let $x \in T$ be a point in the support of $\PF g \mu$,
i.e., for every open neighborhood $V$ of $x$, we have $\PF g\mu(V) > 0$.
Let $\Directed(x)$ denote the set of pairs $(V,B)$ where $V$ is an open
neighborhood of $x$ and $B$ is a measurable subset of $V$ with
$\PF g\mu(B) > 0$, and write $(V,B) \preccurlyeq (V',B')$ when $V \subseteq V'$. 
Note that this relation is a partial ordering on $\Directed(x)$ and makes
$\Directed(x)$ a downward directed set.
We say that $x$ \defn{has the Tjur property (for $\mu$ along $g$)} when
the directed limit of the net $(\mu^{g^{-1}(B)} \st (V,B) \in \Directed(x) )$ exists in the weak topology on the space of probability measures.
\end{definition}

\begin{remark}
We will sometimes write $\mu^{B}_g$ for $\mu^{g^{-1}(B)}$ when it simplifies notation.
We will also write $\mu^x_g$ to denote the directed limit when it exists.
By the Portmanteau Lemma, 
\[
\mu^x_g (A) =
\lim_{(V, B) \in \Directed(x)} \mu_g^{B}(A)
\]
for all $\mu^x_g$-continuity sets $A$.
We will write $\mu^x$ when $g$ is clear from context.
\end{remark}

The usefulness of the Tjur property is demonstrated by the following result, which is a consequence
of Proposition~9.9.1, Corollary~9.9.2, and Proposition~9.10.5 of \cite{MR595868}:
\begin{lemma}
\label{maintjurlemma}
Let $S$ be a completely regular Hausdorff spaces, 
$T$ a metric space,
$\mu$ a Radon probability measure on $S$, and 
$g\colon S\to T$ be a measurable function.

Suppose that there is a $\PF g \mu$-measure one set $C$ such that every 
$x \in C$ has the Tjur property (for $\mu$ along $g$).
For every $x \in C$, let 
$\<B^x_{n}\>_{n\in\Naturals}$ be a sequence of measurable subsets of $T$ 
such that each $B^x_{n}$ is contained in the $2^{-n}$-ball around $x$
and
satisfies $\PF g\mu(B^x_{n}) > 0$.

Then the function $\kappa \colon C \to \ProbMeasures(S)$ given by
\[
\label{Tjureq}
\kappa(x) \defas
\lim_{n \to \infty}
\mu_g^{B^x_n}
\]
for every $x \in C$ (where the limit is taken in the weak topology) is 
continuous and 
can be extended to
a disintegration of $\mu$ along $g$.
\end{lemma}

Many common properties imply that a point is Tjur.
For example, every point of continuity of an absolutely continuous
distribution is a Tjur point.
Also, any isolated point mass (e.g., a point in the support of a discrete
random variable taking values in a discrete space) is a Tjur point.
On the other hand, nonisolated point masses are not necessarily Tjur points.

Tjur points sometimes exist even in nondominated settings.
Let $G$ be a Dirichlet process with an absolutely continuous
base measure $H$ on a computable metric space $S$.  That is, $G$ is a random discrete probability distribution on $S$ such that, for every finite, measurable partition $A_1,\dotsc,A_k$ of $S$, 
the probability vector $(G(A_1),\dotsc,G(A_k))$ is Dirichlet distributed with parameter $(H(A_1),\dotsc,H(A_k))$.
Conditioned on $G$, let $X$ be $G$-distributed (i.e.,
$X$ is a sample from the random distribution $G$).  Then any
point $x \in S$ in the support of $H$ is a Tjur point, yet there does not exist a
conditional density of $G$ given $X$.

In \cref{tjurconverse,Tjur = continuous disintegration on measure one set}
we will see that, under some regularity conditions, 
the existence of an everywhere continuous disintegration is equivalent to every point being Tjur.
However, even when every point is Tjur, and thus there exists a continuous disintegration,
the resulting disintegration may not be computable.  In fact, one of
the main
constructions of \cite{AFR10}
is an example of a conditional distribution for which 
every point is Tjur
and yet no disintegration is a
computable map, even on large measure set.
We will use a related
construction in Section~\ref{continuous}.

\begin{proposition}[{\cite[Prop.~9.14.2]{MR595868}}]
\label{tjurconverse}
Let $S$ and $T$ be completely regular Hausdorff, 
let $\mu$ be a Radon probability measure
on $S$,
let $g \colon S \to T$ be a continuous function,
and let $\nu = \PF g \mu$ be the pushforward.
Suppose 
\[
\xi\colon C &\to \ProbMeasures(S)
\]
is
a continuous mapping,
where $C$ is a subset of $\supp(\nu)$ such that $C$ is a $\nu$-measure one set.
Then the following two conditions are equivalent.
\begin{enumerate}
\item[(i)] 
	For all $x\in C$,
	the conditional distribution $\mu^x_g$ is defined 
	and 
	$\mu^x_g = \xi(x)$.
\item[(ii)] 
        For all $x\in C$, we have $\supp(\xi(x)) \subseteq g^{-1}(x)$,
        and 	
        $\mu = \int_C \xi(x)\, \nu(\dee x)$.
\end{enumerate}
\end{proposition}
We now give conditions for when there is a continuous disintegration on a measure one set.
\begin{proposition}
\label{Tjur = continuous disintegration on measure one set}
Let $S$ and $T$ be completely regular Hausdorff, with $T$ second countable,
let $\mu$ be a Radon probability measure on $S$, 
let $g\colon S \to T$ be a continuous function,
and let $\nu \defas \PF g \mu$.
Further suppose $C$ is a $\nu$-measure one set.
The following are equivalent.
\begin{enumerate}
	\item Every point in $C$ is a Tjur point of $\mu$ along $g$.
	\item $C \subseteq \supp(\nu)$ and there exists a disintegration $\kappa\colon T \to \ProbMeasures(S)$ of $\mu$ along $g$ whose restriction to $C$ is continuous.
\end{enumerate}
\end{proposition}
\begin{proof}
Assume (1). Then the fact that there is a continuous disintegration follows from \cref{maintjurlemma}. 
 
Next assume that (2) holds and that $\kappa$ is a continuous disintegration of $\mu$ along $g$.
As $g$ is continuous, by Proposition~\ref{tjurconverse},
it suffices to verify (ii), taking $\xi = \kappa$.
First note, by the definition of disintegration, 
	\[
		\mu\bigl(g^{-1}(B) \cap A\bigr) = \int_B \kappa(x)(A) \  \PF g \mu (\dee x),
	\]
for all $A \in \Borel{S}$
and $B \in \Borel{T}$.
But as $\nu(C) = 1$, for all 
	$A \in \Borel{S}$ we have 
\[
		\mu(A) = \mu\bigl(g^{-1}(C) \cap A\bigr) = 
	\int_C \xi(y)(A)\  \nu(\dee y).
\]
In particular, $\mu =  \int_C \xi(y)\, \nu(\dee y)$, as required.
Now let $C^* = \{c \in C\st \mu^c(g^{-1}(c)) = 1\}$. It suffices to show $C^* = C$, as $g$ is continuous.
 
Assume, 
towards
a contradiction, 
that there is a $c \in C\setminus C^*$, i.e., such that $\mu^{c}(g^{-1}(c)) < 1$. As $T$ is second countable, there is a countable decreasing collection of closed neighborhoods $\<N_i\>_{i \in \w}$ of $c$ such that $\{c\} = \bigcap_{i \in \w}N_i$. 
It follows that $\<g^{-1}(N_i)\>_{i \in \w}$ is a decreasing countable sequence of sets whose intersection has $\mu^c$-measure strictly less than one, 
and so, for some $j$, it holds that $c \in N_j \subseteq T$ and $\mu^{c}\bigl(g^{-1}(N_j)\bigr) < 1$.
But as $N_j$ is closed, the map which takes $c'$ and returns $\mu^{c'}\bigl(g^{-1}(N_j)\bigr)$ is upper semi-continuous. Therefore there must also be an open set $N'$ with $c \in N' \subseteq T$ such that 
for all $c' \in N' \cap C$ 
we have
$\mu^{c'}(g^{-1}(N_j)) < 1$. Let $N^*$ be the interior of $N_j$ which is non-empty as $N_j$ is a closed neighborhood of $c$. Then 
\[N'' \defas N' \cap N^* \cap C \subseteq C \setminus C^*.\] 

Because $N'' \subseteq \supp(\nu)$ is a non-empty open set, $\nu(N'') > 0$.  But then
\[
\nu(N'') 
= \int_{N''} \nu(\dee y)
> \int_{N''} \mu^y\bigl( g^{-1}(N'') \bigr) \ \nu(\dee y) = \mu (g^{-1}(N'')) = \nu(N''),
\]
a contradiction.
\end{proof}

Let $\pi_2 : S \times T \to T$ denote the projection onto $T$. 
We now define the space of measures we will use to study the computability of disintegration.

\begin{definition}\label{ACD}
For second countable regular Hausdorff spaces $S$ and $T$,
let $\ACD_{S, T} \subseteq \ProbMeasures(S \times T)$ be the subset consisting of those measures
$\mu$ such that
\begin{enumerate}
\item the pushforward of $\mu$ along $\pi_2$, has full support, i.e., $\supp (\PF {\pi_2} \mu) = T$, and
\item $\mu$ admits a (necessarily unique, by \cref{contdistareunique})
continuous disintegration along the projection map $\pi_2 \colon S\times T \to T$.
\end{enumerate}
\end{definition}

It follows from
\cref{Tjur = continuous disintegration on measure one set} 
that $\mu \in \ACD_{S, T}$ if and only if every point of $T$ is a Tjur point for $\mu$ along $\pi_2$. 

Because we are conditioning on a projection map,
a disintegration can be identified with a
continuous maps $\ContFuncs(T, \ProbMeasures(S))$.
Let $\iota_1\colon  \ProbMeasures(S \times T) \to \ProbMeasures(S)$ be the map defined by
$\iota_1(\mu)(A) = \mu(A \times T)$
for all $\mu \in \ProbMeasures(S\times T)$ and $A \in \Borel{S}$. 
It is easy to check that $\iota_1$ is a continuous map.  Let $\DD_{S, T}\colon \ACD_{S, T} \rightarrow \ContFuncs(T, \ProbMeasures(S))$ be
defined by $\DD_{S,T}(\mu) = \iota_1 \circ \kappa$, where $\kappa \colon T \to \ProbMeasures(S\times T)$ is the (unique) continuous disintegration of $\mu$ along $T$. 
The map $\iota_1$ is injective on the image of $\ACD_{S, T}$ under $\DD_{S,T}$, and so when no confusion can arise, we will also call $\DD_{S, T}(\mu)$ the disintegration of $\mu$ when $\mu \in \ACD_{S,T}$.

One often requires that a disintegration be merely a.e.\ continuous (rather
than continuous everywhere). However, for any measure admitting 
an a.e.\ continuous disintegration, there is a $G_\delta$ subset of $T$ of
measure one on which it is continuous everywhere. 
One could therefore consider all measures on $S \times T$ such that the set of Tjur points contains a $G_\delta$-subset of $T$ of measure one. However one would then have to define the domain of the disintegration map to be the set of pairs consisting of a measure on $S \times T$ along with a $G_\delta$-subset of $T$. In this context the fundamental results of this paper should still hold with essentially the same proofs;
however,
the notational complexity would also be greatly increased.

\subsection{Weaker conditions than the Tjur property}
\label{weakerFN}

In Section \ref{specific} we will consider a distribution which does not admit a continuous disintegration and is thus not in $\ACD_{S, T}$. To make sense of this we need a weaker notion of disintegration than that given by Tjur. These definitions are based on those from Fraser and Naderi \cite{MR1462405}.

\begin{definition}
\label{converging regularly}
Suppose 
$T$ is a separable metric space
and 
$\nu$ is a probability measure on $T$.
A sequence of sets $\<E_n\>_{n \in \w}\in \Borel{T}$ is said to \defn{converge regularly} to $x$ (with respect to $\nu$) 
if there is a sequence of closed balls $\<B_n\>_{n \in \w}$ 
of respective radii $\<r_n\>_{n \in \w}$ such that
\begin{enumerate}
\item $\lim_{n \to \infty} r_n = 0$,
\item $x \in E_n \subseteq B_n$ for all $n$, and
\item there is an $\alpha > 0$ such that  $\mu(E_n) \geq \alpha \cdot \mu(B_n)$ for all $n$.
\end{enumerate} 
\end{definition}

Intuitively, a sequence of sets converges regularly to $x$ if it is ``close'' to a decreasing sequence of closed
balls whose intersection is $\{x\}$.

\begin{definition}
For $F \in \Borel{T}$,
a class $W \subseteq \Borel{T}$ is a \defn{Vitali cover of $F$} (with respect to $\nu$)
when, for all $x \in F$, 
there is a sequence of elements of $W$ that converges regularly to $x$ with respect to $\nu$.
\end{definition}

\begin{definition}
\label{Vitali covering property}
A class $V \subseteq \Borel{T}$ has the \defn{Vitali covering property} (with respect to $\nu$)
when, for every $F \in \Borel{T}$, (1) there is a Vitali cover $W \subseteq V$ of $F$ with respect to $\nu$ 
and (2) for every Vitali cover $W \subseteq V$ of $F$ with respect to $\nu$,
there is a collection $\{F_n \st n \in \w\}\subseteq W$ of disjoint sets 
such that $\mu(F  \setminus \bigcup_{n \in \w} F_n) = 0$.
\end{definition}
These definitions allow us to define a notion of disintegration at points:
\begin{proposition}[{\cite[Thm.~2]{MR1462405}}]
\label{Fraser-disintegration thm}
Let $g\colon S \to T$, let $\mu \in \ProbMeasures(S)$,
and let $\nu = \PF g \mu$. Suppose that $V\subseteq \Borel{T}$ has the Vitali covering property with respect to $\nu$. 
\begin{itemize}
	\item[(i)] 
	         For every $A \in \Borel{S}$ there exists $C_A \in \Borel{T}$ with $\nu(C_A) = 1$ such that, 
	                   for every $x\in C_A$
		                and every sequence $\<E_n\>_{n \in \w}\subseteq V$ that converges regularly to $x$, the limit 
\[
\lim_{n \to \infty} \frac{\mu(g^{-1}(E_n) \cap A)}{\mu(g^{-1}(E_n))}
\]
exists 
and
is independent of the choice of $\<E_n\>_{n\in\w}$.
Let $\mu_{g,V}^x (A)$ denote this limit when it exists and define $\mu_{g,V}^x (A) = \mu(A)$ otherwise.
\item[(ii)] The map $x \mapsto \mu_{g,V}^x(A)$ is measurable with respect to the $\nu$-completion of $\Borel{T}$ and, for each set $E$ in the $\nu$-completion of $\Borel{T}$,
\[
\mu(g^{-1}(E) \cap A) = \int_E \mu_{g,V}^x(A) \ \nu(\dee x).
\]
\end{itemize}
\end{proposition}
\begin{remark}
By \cref{maintjurlemma}, 
note that if $x \in T$ is a Tjur point for $\mu$ along $g$, then $\mu_{g,V}^x = \mu^x_g$, and so the left hand side does not depend on $V$.
\end{remark}
\begin{remark}
	\label{meas-map-rmk}
Because $S$ is assumed to be a 
separable metric space, we may arrange things 
so that $A \mapsto \mu_{g,V}^x(A)$ is a probability measure for every $x$ and that $x \mapsto \mu_{g,V}^x$ is a measurable map.
(In other words, $(x, A) \mapsto \mu_{g,V}^x(A)$ is a probability kernel.)
\end{remark}

The notion of disintegration given by \cref{Fraser-disintegration thm} relies heavily on the metric structure of the underlying set. Fortunately, though, there are many situations when there are natural collections with the Vitali property. 
For example, the class of closed sets in the Lebesgue completion of the Borel subsets of $\Reals^k$ has the Vitali covering property with respect to Lebesgue measure (see \cite[p.~304]{MR1369805} for references to proofs).
For our purposes, the following result suffices.
\begin{proposition}
\label{ultrametric-has-Fraser}
If $T$ is an ultrametric space then the collection of open balls has the Vitali covering property with respect to any measure on $T$. 
\end{proposition}
\begin{proof}
This follows from the fact that every open set is the disjoint union of balls and every open ball is clopen. 
\end{proof}
In particular, if the underlying space is a countable product of Cantor space, Baire space, and discrete spaces then the collection of open balls has the Vitali covering property with respect to any measure. 

It is worth contrasting \cref{TjurProperty} with \cref{converging regularly}. In particular, the notion of a Tjur point can be thought of as requiring that the conditional measures converge at each point irrespective of the manner in which we approximate the point, and specifically does not require the measurable sets in the approximating sequence
to actually contain the point they are approximating, so long as each is contained in an open neighborhood that does. In contrast, the use of regular convergence by Fraser and Naderi is in terms of a fixed class of sets having the Vitali covering property with respect to the given measure,
and does require that the sets one is conditioning on contain the point being approximated.

\section{Represented spaces and Weihrauch reducibility}
\label{Wred}

We now introduce the notion of a space equipped with a representation and describe some special cases in detail, including spaces of probability measures.  We then define some key operators and introduce the notion of a Weih\-rauch degree in order to
compare the computational complexity of operators.
For background on represented spaces, see \cite{arno}.  A similar treatment can be found in 
\cite{Col09} and \cite{Col10}.  Both build off the framework of synthetic topology \cite{Esc04}.
For  background on Weihrauch reducibility and Weihrauch degrees,
see, e.g., \cite{MR2791341},
or the introduction of \cite{Brattka2015249}.

\subsection{Represented spaces}

In order to define a notion of computability beyond that of the classical setting of Baire space, $\Baire$,
we adopt the framework of represented spaces, which is itself a distillation of core ideas from the computable analysis framework connecting back
to Turing's original work on computable real numbers.
The following definitions are taken from \cite{arno}.

\begin{definition}
A \defn{represented space} is a set $X$ along with 
a partial surjective function
$\delta_X:\subseteq\Baire \rightarrow X,$
called a \defn{representation}.
A point $x \in \Baire$ in the domain of $\delta_X$ is said to be a \defn{name} for its image $\delta_X(x)$ in $X$.
\end{definition}

In general, every point in $X$ will have many names:
indeed, each represented space $(X, \delta_X)$ defines an equivalence relation $\equiv_X$ on 
$\delta_X^{-1}(X) \subseteq \Baire$
given by $p \equiv_X q$ iff $\delta_X(p) = \delta_X(q)$.
For each equivalence relation on a subset of $\Nats^\Nats$, there is a corresponding represented space.

Functions between represented spaces can be understood as functions mapping names in the domain to names in the codomain.

\begin{definition}
Let $(X, \delta_X)$ and $(Y, \delta_Y)$ be represented spaces and let
$f:\subseteq X \multito Y$ be a multi-valued partial function. Then a
\defn{realizer} for $f$ is a (single-valued) partial function $F: \subseteq \Baire \to \Baire$ such that
for all $p$ in the domain of $f\circ \delta_X$, 
\[
\delta_Y\bigl(F(p)\bigr) \in f \bigl(\delta_X(p)\bigr).
\]
When this holds, we write $F \vdash f$.
\end{definition}

Note that a function $F$ is a realizer 
for a single-valued function
if and only if it preserves the equivalence relations associated to the represented spaces.

In what follows we will give explicit representations for only a few fundamental spaces. It is worth mentioning that what is important is not the specific representations but their defining properties: if the defining properties are computable in the appropriate senses, then any two representations that satisfy the defining properties will be computably isomorphic.
As such, after introducing only those basic representations that we need to get started, 
we will aim to avoid mentioning or using the details of a given representation whenever possible. 
For more on this approach to represented spaces see \cite{arno} and references therein.

\subsubsection{Continuous functions}
A map between represented spaces is called continuous (computable) if it has a continuous (computable) realizer. 
The notions of 
continuous maps between represented spaces and 
continuous maps between topological spaces are, in general, distinct, and should not be conflated.
Once we have introduced the notion of an admissible represented space, however,
\cref{thmad} relates these two notions precisely.
\begin{definition}
	\label{contmap-def}
Let $X$ and $Y$ be represented spaces.
Then the space $\ContFuncs(X,Y)$ of continuous maps between $X$ and $Y$ can itself be made into a represented space. One such representation of $\ContFuncs(X,Y)$, which we will use in this paper, is given by $\delta(0^n1p) = f$ iff the $n$'th oracle Turing machine computes a realizer for $f$ from oracle $p\in\Cantor$.
\end{definition} 
Intuitively speaking, the name of a map $f$ describes an oracle Turing machine that translates names for inputs in $X$ to names for outputs in $Y$.
Natural operations (evaluation, currying, uncurrying, composition, etc.) on continuous functions are computable (see \cite[Prop.~3.3]{arno}).

\subsubsection{The Naturals, Sierpinski space, and 
spaces of open sets}
The representation of sets and relations requires that we introduce two special represented spaces, namely $\Nats = (\Nats,\delta_\Nats)$ and $\Sier = (\{0,1\},\delta_\Sier)$.
In particular, $\delta_\Nats(0^n10^\omega) = n$ for all $n \in \Nats$, $\delta_\Sier(0^\omega) = 0$, and $\delta_\Sier(p) = 1$ for $p \neq 0^\omega$.
Note that $\Sier$ should not be confused with the represented space of binary digits $\ZO = (\{0,1\},\delta_\ZO)$ where $\delta_\ZO(0^\omega) =0$ and $\delta_\ZO(1^\omega) = 1$.  Indeed, it is common to define $\Sier$ on the two-point space $\{\bot,\top\}$ instead of $\{0,1\}$, however, the latter has the advantage, as we will see, that the EC operator becomes an identity map.

The following basic operations on $\Sier$ are computable (see
	{\cite[Prop.~4.1]{arno}}
for details):
\begin{enumerate}
\item finite and (i.e., min), $\wedge\colon \Sier \times \Sier \to \Sier$;
\item finite or (i.e., max), $\vee\colon \Sier \times \Sier \to \Sier$;
\item countable or (i.e., max), $\bigvee\colon \ContFuncs(\Nats,\Sier) \to \Sier$.
\end{enumerate}

Given Sierpinski space, we can define (computable) open sets as those for which membership $\in \colon X \to \Sier$ is continuous (computable).
Note that, in light of the representation of $\Sier$, this corresponds to semi-decidability relative to the name of the point in $X$ and the open set.
\begin{definition}
Let $X$ be a represented space.  Then the represented space $\Opens(X)$ of open sets of $X$ is identified with the represented space $\ContFuncs(X,\Sier)$.  In particular, a map $f \in \ContFuncs(X,\Sier)$ is taken to represent the inverse image $f^{-1}(1)$.
\end{definition}
In order to connect this notion with classical computability, note that the computable elements of $\Opens(\Nats)$ are the computably enumerable subsets of $\Nats$.

\begin{lemma}[{\cite[Prop.~4.2]{arno}}]
\label{arno-Prop-6}
Let $X$ and $Y$ be represented space. Then the following operations are computable:
\begin{enumerate}
\item finite intersection, $\cap \colon \Opens(X) \times \Opens(X) \to \Opens(X)$;
\item finite union, $\cup \colon \Opens(X) \times \Opens(X) \to \Opens(X)$;
\item countable union, $\bigcup \colon \ContFuncs(\Nats,\Opens(X)) \to \Opens(X)$;
\item inverse image, $^{-1} \colon \ContFuncs(X,Y) \to \ContFuncs(\Opens(Y),\Opens(X))$;
\item membership, $\in\colon X \times \Opens(X) \to \Sier$; 
\item product, $\times\colon \Opens(X) \times \Opens(Y) \to \Opens(X\times Y)$. 
\end{enumerate}
\end{lemma}

The notion of admissibility connects topological continuity and continuity in the sense of represented spaces.
For more details on represented spaces and admissible representations, see \cite[Ch.~3]{MR1795407} or \cite[\S9]{arno}, which is based on earlier work in \cite[\S4.3]{Schroder-phd} and \cite{MR1948058}.

By part (4) of \cref{arno-Prop-6}, the map $f \mapsto f^{-1}$ is computable.
Admissibility can be characterized in terms of the inverse of this map.
\begin{definition}\label{defad}
A represented space $X$ is \defn{admissible (computably admissible)}
when
the map $f \mapsto f^{-1} \colon \ContFuncs(Y,X) \to \ContFuncs(\Opens(X),\Opens(Y))$ has a well-defined and continuous (computable) partial inverse for
any represented space $Y$.
\end{definition}
\begin{remark}
The previous definition makes use of the following characterization of admissibility, which can be found in \cite[Thm.~9.11]{arno}.
\end{remark}

\begin{proposition}[{\cite[Thm.~9.11]{arno}}]\label{thmad}
A represented space $X$ is admissible if and only if any topologically continuous function $f \colon Y \to X$ (i.e., one for which the map $f^{-1} \colon \Opens(X) \to \Opens(Y)$ is well-defined) is 
continuous as a function between represented spaces (i.e., $f \in \ContFuncs(Y,X)$).
\end{proposition}

Every representation that we use in this paper is computably admissible.

\subsubsection{Some standard spaces as represented spaces}

We are interested in studying operators defined on certain common mathematical spaces. We first describe how these spaces can be made into computably admissible represented spaces.
In particular, the Naturals $\Nats$, Sierpinski space $\Sier$, the binary digits $\ZO$, but also the Reals $\Reals$ can all be made into computably admissible spaces with respect to their standard topologies.  
Finite or countable products of these spaces are again computably admissible.  
(The details are not important for our presentation, but the interested reader can find a thorough account in \cite{MR1795407}.)

For the represented space $\Reals$ of reals we will require that the map 
$\id_{\Rationals}\colon \ZO \times \Nats \times \Nats \to \Reals$ 
defined by
$s, a,b \mapsto (-1)^{s} \cdot a/b$ 
is computable, as well as the $<$ and $>$ comparison operators $\Reals \times \Reals \to \Sier$.  
These requirements on the represented space $\Reals$ already determine several other key properties, 
including the computability of standard operations on the reals such as addition, subtraction, multiplication, division by a nonzero, and exponentiation.

We will need to define two slightly more exotic represented spaces.  We define $\LReals$ to be the represented space with underlying set $\Reals \cup \{\infty\}$, such that the comparison map $< \colon \Reals \times \LReals \to \Sier$ is computable, and which is computably admissible with respect to the 
right-order topology 
generated by the rays 
$\{ (x,\infty] \colon x \in \Reals \}$.
Similarly we define $\UReals$  to be the represented space with underlying set $\Reals \cup \{-\infty\}$, such that the comparison map $> \colon \Reals \times \UReals \to \Sier$ is computable, and which is computably admissible with respect to the 
left-order topology 
generated by the rays 
$\{ [-\infty,x) \colon x \in \Reals \}$.
Note that an element of $\LReals$ is a computable point precisely when it is a lower semi-computable real. We will also use $\Reals^+, \LReals^+, \UReals^+$ for the subspaces with underlying set $(0, \infty]$ and similarly $\Reals^-, \LReals^-, \UReals^-$ for the subspaces with underlying set $[-\infty, 0)$.

The following standard operations are computable:
\begin{enumerate}
\item identity, $\id \colon \Reals \to \LReals$ and $\id \colon \Reals \to \UReals$, and also $\id \colon \LReals \times \UReals \to \Reals$ defined on the diagonal $(x,x)$;
\item comparison, $> \colon \LReals \times \UReals \to \Sier$ and $< \colon \UReals \times \LReals \to \Sier$;
\item addition, $+\colon \LReals \times \LReals \to \LReals$;
\item negation, $-\colon \LReals \to \UReals$ as well as $-\colon \UReals \to \LReals$;
\item multiplication by a positive real, $\times_+\colon \LReals \times \Reals^+ \to \LReals$;
\item multiplication by a negative real, $\times_-\colon \LReals \times \Reals^- \to \UReals$;
\item reciprocation, $^{-1} \colon \LReals^+ \to \UReals^+$, $^{-1} \colon \LReals^- \to \UReals^-$,  $^{-1} \colon \UReals^+ \to \LReals^+$, and $^{-1} \colon \UReals^- \to \LReals^-$ for nonzero quantities;
\item infinite sums, $\sum\colon \ContFuncs(\Nats, \LReals^+) \to \LReals$. 
\end{enumerate}

Besides these concrete spaces, we are interested in admissible representations of complete separable metric spaces.

\begin{definition}\label{cmsdefn}
A \defn{computable metric space} is a triple $((S,\delta_S),d_S,s)$ such that 
$(S,d_S)$ is a complete separable metric space, 
$(S,\delta_S)$ is an admissible represented space, 
$s \in \ContFuncs(\Nats,S)$ is a computable and dense sequence in $S$,
and the following are computable:
\begin{enumerate}
\item the distance function, $d_S\colon S \times S \to \Reals$;
\item the limit operator
	$\ContFuncs(\Nats,S) \to S$ defined
	on the set of rapidly converging Cauchy sequences, 
	i.e., defined on
		$\<r_n\>_{n\in\Nats} \in \ContFuncs(\Nats,S)$ 
	satisfying
	$| d_S(r_n) - d_S(r_{n+1})| < 2^{-n}$ for all $n\in\Nats$; 
\item the map that takes a sequence $\<(s_n, q_n)\>_{n \in \Nats} \in \ContFuncs(\Nats, \Nats \times \Rationals)$ to the element $\bigcup_{ n \in \Nats} \Balls(s_n, q_n) = \{x\in S \st (\exists n \in \Nats)\ d_S(x, s_n) < q_n\} \in \Opens(S, \delta_S)$ is computable with a computable multi-valued inverse. 
\end{enumerate}
\end{definition}
We will omit mention of $\delta_S$ and refer to $(S, \delta_S)$ by $S$ when no confusion can arise. 

Notice that we do not insist that the limit map in the definition of a space be computable on the collection of all Cauchy sequences, but only those sequences which are rapidly converging. In particular, all such rapidly converging Cauchy sequences have a concrete, and hence computable, bound (which tends to zero) on the distance between each element of the sequence and the limit point. We will reconsider the general limit operator on Cauchy sequences in Sections~\ref{the-EC-operator} and~\ref{limopsection}.

Note that $(\Reals, d, \Rationals)$, where $d(x,y) = |x - y|$, is a computable metric space. Another important computable metric space is Cantor space, \Cantor, by which we mean the computable metric space consisting of the space of functions
\[
\Nats \to \{0, 1\}
\]
with the usual ultrametric distance, and with dense set consisting of the
eventually zero functions enumerated lexicographically.

\subsubsection{Representing probability measures}

\begin{definition}
Let $(S,d_S)$ be a metric space.
For $\varepsilon > 0 $ and $A \subseteq S$, let $A^\varepsilon \defas \{x \st d_S(x,A)<\varepsilon\}$.
The \defn{Prokhorov metric} $\pmetric$ on $\ProbMeasures(S)$ is defined by
\[
\pmetric(\mu, \nu) 
\defas
\inf\bigl\{\varepsilon > 0
\st
\mu(A) \le \nu(A^\varepsilon) + \varepsilon \ \text{~for every Borel set~} A
\bigr\}
\]
for $\mu, \nu \in \ProbMeasures(S)$.
\end{definition}

Note that the Prokhorov metric generates the weak topology on $\ProbMeasures(S)$.

The following result shows how to make the space $\ProbMeasures(S)$ of Borel probability measures on a computable metric space $S$ into a computable metric space itself.

\begin{proposition}
	[{\cite[Appendix~B.6]{MR2159646}}]
Let $(S,d_S, s)$ be a computable metric space.
Then the space of Borel probability measures $\ProbMeasures(S)$ is itself a computable metric space under the Prokhorov metric
and with dense set
	\[
\left\{\sum_{i=0}^{k} q_i \delta_{s(m_i)} \colon k, m_0, \dots, m_k \in \Nats,\ q_0, \dots, q_k \in \Rationals^+ \text{ and } \sum_{i = 0}^k q_i = 1 \right\},
\]
\end{proposition}
where $\Rationals^+ \defas \{q \in \Rationals\st q > 0\}$.
An alternative representation of a Borel probability measure is in terms of its restriction to the class of open sets. Such a set function is known as a \emph{valuation}. A valuation on a separable metric space uniquely determines a Borel probability measure \cite[\S3.1]{MR2351942}.

\newcommand{\Valuations}{\mathcal V}
\begin{definition}
Let $X$ be an admissible represented space.
The represented space $\Valuations(X)$ of \defn{valuations on $X$} is the subspace of $\ContFuncs(\Opens(X),\LReals)$ corresponding to restrictions of Borel probability measures to the open sets.
\end{definition}

By {\cite[Thm.~3.3]{MR2351942}},
$\Valuations(X)$ is computably admissible when $X$ is.
The next result relates the representation of Borel measures as points in the Prokhorov metric to their representation as points in the space of valuations.

\begin{lemma}
\label{Lemma: id map from prob to val is continuous}
Let $X$ be a computable metric space.
The identity map from $\ProbMeasures(X)$ to $\Valuations(X)$ is computable and so is its inverse.
\end{lemma}
In other words,
a probability measure is computable
if and only if, uniformly in the name for an open set, we can lower semi-compute the measure of that open set. See \cite[Sec.~2.4]{MR2410914}  or \cite[Thm.~4.2.1]{MR2519075} for more on computable measures.

Let $\mu$ be a probability measure on a topological space $S$, 
let $E$ be a Borel subset of $S$, 
and
let $\partial E$ denote the boundary of $E$, i.e., the difference 
between
its closure and its interior.  
The set $E$ is a \defn{$\mu$-continuity set} when $\mu( \partial E) = 0$.

For a computable 
metric space $X$
and Borel probability measure $\mu \in \ProbMeasures(X)$, 
let $\AD {\mu}(X)$ be the class of open sets $U \in \Opens(X)$ that are also $\mu$-continuity sets.
Concretely, one representation for $\AD {\mu}(X)$, which we will use here,
is the map $\delta \colon \Baire \to \AD {\mu}(X)$ such that $\delta(p) = U$ if and only if $p=(u,v)$, where $u$ is a name for $U$ as an element of $\Opens(X)$ and $v$ is a name an open set $V$, as an element of $\Opens(X)$, 
where $V \subseteq X \setminus U$ and $\mu(V) =\mu(X \setminus U)$. 
A more elegant approach, which we do not take here, is to define $\AD {\mu}(X)$ in terms of the extremal
property that containment (i.e., the $\in$ relation) is computable on a $\mu$-measure one set.
We now describe several computable operations:
\begin{lemma}
	\label{adcompop}
	Let $X$ and $Y$ be computable metric spaces, and let $\mu, \mu_X \in \ProbMeasures(X)$ and $\mu_Y \in\ProbMeasures(Y)$.
	The following maps are computable:
\begin{enumerate}
\item $\cap \colon \AD {\mu}(X) \times \AD {\mu}(X) \to \AD {\mu}(X)$;
\item $\cup \colon \AD {\mu}(X) \times \AD {\mu}(X) \to \AD {\mu}(X)$;
\item $\times \colon \AD {\mu_X}(X) \times \AD {\mu_Y}(Y) \to \AD {\mu_X \otimes \mu_Y}(X \times Y)$;
\item $\icmu {{}} \colon \AD {\mu}(X) \multito \AD {\mu}(X)$, which takes an open $\mu$-continuity set $U$ to 
the set of open $\mu$-continuity sets $V$ contained within, and of the same $\mu$-measure as, 
$X \setminus U$;
\item $\in \colon X \times \AD {\mu}(X) \to \Sier$;
\item $\mathrm{id}: \AD {\mu}(X) \to \Opens(X)$.
\end{enumerate}
\end{lemma}
\begin{proof}
	That (1), (2), and (5) are computable follows immediately from \cref{arno-Prop-6}, parts (1), (2), and (5), respectively.
	The computability of (6) follows trivially from the definition of the representation for $\AD {\mu}(X)$.
	The computability of (4) follows from the fact that, if $p=(u,v)$ is a name for $U$, then $(v,u)$ is a name for an element in $\icmu{U}$. 
To see that the map (3) is computable 
let $U$ and $V$ be a $\mu_X$- and $\mu_Y$-continuity set, respectively, and compute $U^* \in \icmus{U}{\mu_X}$ and $V^* \in \icmus{V}{\mu_Y}$. 
Then $(U^* \times V) \cup (U \times V^*) \cup (U^* \times V^*) \in \icmus{(U \times V)}{(\mu_X \otimes \mu_Y)}$
and is a computable element of $\Opens(X \times Y)$ 
by parts (2) and (6)
	of \cref{arno-Prop-6}.
\end{proof}

\begin{lemma}\label{compmeasureofad}
	Let $X$ be a computable metric space and let $\mu \in \ProbMeasures(X)$.
	The map taking $\mu$ and $U\in \AD {\mu}(X)$ to $\mu(U) \in \Reals$ is computable.
\end{lemma}
\begin{proof}
	Let $U \in \AD {\mu}(X)$.  Then $U$ and some $V \in \icmu U$,
	are elements of $\Opens(X)$ computable from $U$
	by \cref{adcompop}, parts (4) and (6).
	Therefore $\mu(U) \in \LReals$ and $1-\mu(V) \in \UReals$ are computable from $U$ and $V$, respectively.  But $\mu(U) = 1 - \mu(V)$, hence $\mu(U) \in \Reals$ is computable from $U$. 
\end{proof}

\begin{lemma}[{\cite[Lem.~2.15]{MR2410914} and \cite[Lem.~5.1.1]{MR2519075}}]
\label{lem-boss}
Let $(S, d_S, \<s_i\>_{i\in\Nats})$ be a computable metric space, 
let $\mu\in \ProbMeasures(S)$ be a Borel probability measure on $S$, 
and let $B_\mu$ be the set of all sequences 
$\<\epsilon_i\>_{i\in \Nats}$ of positive reals 
such that 
\[
\bigl\{\Balls(s_j, \epsilon_i)\bigr\}_{i,j\in\Nats}
\]
forms a subbasis for $S$ comprised of $\mu$-continuity sets.
Then the multi-valued map from $\ProbMeasures(S) \multito \ContFuncs(\Nats,\Reals_+)$ taking $\mu$ to $B_\mu$ is computable.
\end{lemma}

\begin{definition}
\label{almostdec-def}
Let $X$ be a computable metric space, let $\mu$ be a Borel probability measure on $X$,
and let $\CB {\mu}(X) \subseteq \ContFuncs(\Nats,\AD {\mu}(X))$ be the represented space of enumerations
 of countable bases of $X$ composed of $\mu$-continuity sets, 
where
$\delta(p) = B$ if and only $p=(p',p'')$ where $p'$ is a name for $B$ as an element of $\ContFuncs(\Nats,\AD {\mu}(X))$ and $p''$ is a name for the (multi-valued) right inverse of
the map
\[
{\textstyle\bigcup}^{B} \colon \ContFuncs(\Nats, \Nats) \to \Opens(X)
\]
taking $U$ to $\bigcup_{n \in \Nats} B(U(n))$.
Elements of $\CB {\mu}(X)$ are called $\mu$-continuous bases.
\end{definition}

Note that the map $\bigcup^{B}$ is necessarily computable from $B$ by \cref{arno-Prop-6}, part (3).
We have the following important corollary, first proved by Hoyrup and Rojas \cite[][Cor.~5.2.1]{MR2519075} in a non-relativized form.
\begin{corollary}
\label{measure-computable from basis}
Let $X$ be a computable metric space and let $B \in \CB {\mu}(X)$ be a $\mu$-continuous basis.
Then $\mu$ is computable from $\<\mu(\bigcup_{n \in I}B(n)\>_{I \subseteq \Nats, |I|< \w}$ and $B$.
\end{corollary}
\begin{proof}
By a relativized version of \cite[Thm.~4.2.1]{MR2519075} it suffices to show that for any open set $U \in \Opens(X)$, 
we can lower semi-compute $\mu(U)$ from $\<\mu(\bigcup_{n \in I}B(n))\>_{I \subseteq \Nats, |I|< \w}$, 
$U$ and $B$.
By definition, we can compute, from $U$ and $B$, a function
$f_U\in\ContFuncs(\Nats, \Nats)$ such that $U = \bigcup_{n \in \Nats} B(f_U(n))$.
By the continuity of measures, 
$\mu(U) = \sup_{n \in \Nats} \mu(\bigcup_{i \leq n} B(f_U(i))$, which is an increasing sequence that is computable from $\<\mu(\bigcup_{n \in I}B(n))\>_{I \subseteq \Nats, |I|< \w}$, $B$, and $U$.
\end{proof}

\begin{lemma}\label{continuitysubbasis}
Let $T$ be a computable metric space.
The multi-valued map taking a Borel probability measure $\mu \in \ProbMeasures(T)$ 
to the set $\CB{\mu}(T)$ of $\mu$-continuous bases of $T$ is computable.
\end{lemma}
\begin{proof}
Let $\tmetric$ denote the computable metric on $T$.
By Lemma~\ref{lem-boss},
we can compute, from $\mu$, 
a dense collection $\<\epsilon_{k}\>_{k \in \Nats}$ of positive reals
such that $\< \Balls(t_i, \epsilon_k)\>_{i, k \in\Nats}$ is a subbasis composed of $\mu$-continuity sets. 
Let $V(i, k)$ be the interior of the complement of the open ball $\Balls(t_i, \epsilon_k)$.
For all $b \in T$, it holds that $b \in V(i,k)$ if and only if there are $j,k' \in \Nats$ such that 
\[
\tmetric(t_j, b) < \epsilon_{k'} \quad \text{and} \quad
\tmetric(t_{p_0(n)},t_j) > \epsilon_{p_1(n)} + \epsilon_{k'},
\] 
i.e., $b \in \Balls(t_j,\epsilon_{k'}) \subseteq V(i,k)$. 
The set of all such pairs is computably enumerable from $\mu$, uniformly in $i,k$, 
and so $V(i,k)$, viewed as an element of $\Opens(T)$, is computable from $\mu$, uniformly in $i,k$. 
Therefore, $\Balls(t_i, \epsilon_k)$, viewed as an element of $\AD {\mu}(T)$, is computable from $\mu$, uniformly in $i, k$.
Finite intersection is a computable operation on $\AD {\mu}(T)$ (by part (1) of \cref{adcompop}),
and so we can compute an enumeration $B \in \ContFuncs(\Nats, \AD {\mu}(T))$ of a basis for $T$ from $\mu$.

Finally, by the definition of a computable metric space, from every $U \in \Opens(T)$, we can compute a collection of balls of rational radii whose union is $U$, and as $\<\epsilon_k\>_{k \in \Nats}$ is dense and computable from $\mu$, 
it is clear that the right-inverse of $\bigcup^B$ is computable from $\mu$, and so $B$, viewed as an element of $\CB{\mu}(T)$, is computable from $\mu$.
\end{proof}

\subsection{Weihrauch reducibility}

Let $\pairing{\cdot,\cdot}\colon \Baire \times \Baire \to \Baire$ be a computable bijection with a computable inverse,
and let
\[
\id \colon \Baire \to \Baire
\]
denote the identity map on $\Baire$. For $f,g \colon \Baire \to \Baire$, 
write $\pairing {f,g}$ to denote the function mapping $x \in \Baire$ to $\pairing {f(x),g(x)} \in \Baire$.

\begin{definition}
	\label{Weihrauch-defone}
Let $\mathcal{F}$ and $\mathcal{G}$ be sets of partial functions 
from $\Baire$ to $\Baire$. 
Then $\mathcal{F}$ is \defn{Weihrauch reducible} to $\mathcal{G}$, written $\mathcal{F} \leW \mathcal{G}$, when there are 
computable partial functions $H, K:\subseteq \Baire \to \Baire$ such that
\[
(\forall G \in \mathcal{G})(\exists F \in \mathcal{F}) \  F = H \circ \pairing{ \id, G \circ K }.
\]
The set
$\mathcal{F}$ is \defn{strongly Weihrauch reducible}
to $\mathcal{G}$, written $\mathcal{F}\lesW \mathcal{G}$ when there are computable
partial functions $H, K:\subseteq  \Baire \to \Baire$ such that
\[
(\forall G \in \mathcal{G})(\exists F \in \mathcal{F}) \  F = H \circ G \circ K.
\]
\end{definition}

Recall that $F \vdash f$ means that $F :\subseteq \Baire \to \Baire$ is a realizer for a map $f$ between represented spaces.

\begin{definition}
	\label{Weihrauch-deftwo}
Let $f, g$ be multi-valued functions on (not necessarily the same)
represented spaces. Then $f$ is 
\defn{Weihrauch reducible} to
$g$, written $f \leW g$, when 
\[
\{F \st F \vdash f\} \leW \{G \st G \vdash g\}.
\]
Two functions $f$ and $g$ are 
Weihrauch equivalent (written $f \eqW g$) when $f \leW g$ and $g \leW f$.
	Weihrauch reducibility is transitive, and 
the  \defn{Weihrauch degree} of a function is the $\eqW$-class of the function.

Similarly, $f \lesW g$ when $\{F \st F \vdash f\} \lesW \{G \st G \vdash g\}$, and the \defn{strong Weihrauch degree} of a function is its $\eqsW$-class.
\end{definition}

When a function $f$ has the property that $\textrm{id} \times f \eqsW f$, where $\textrm{id}$ is the identity on $\Baire$, the function $f$ is 
called
a \defn{cylinder}. We will make use of the following fact on several occasions.

\begin{lemma}[{\cite[Cor.~3.6]{MR2791341}}]
	\label{cylinder-W-sW-lemma}
Let $f$ and $g$ be multi-valued functions on represented spaces, and suppose that $f$ is a cylinder. If $g \leW f$, then $g \lesW f$.
\end{lemma}

\subsection{Operators}
\subsubsection{The disintegration operator}

For computable topological spaces $S$ and $T$, recall the disintegration
map $\DD_{S, T}\colon \ACD_{S, T} \to \ContFuncs(T, \ProbMeasures(S))$.
In what follows, we will write $\DD$ for 
$\DD_{[0,1], [0,1]}$.
One of our main theorems is a characterization of the strong Weihrauch degree of $\DD$,
showing that $\DD \eqsW \Lim$.  Before formally introducing the operator $\Lim$, we first describe an equivalent operator $\EC$.

\subsubsection{The $\EC$ operator}
\label{the-EC-operator}

The $\EC$ operator can be thought of as taking an enumeration of a set of
natural numbers to its characteristic function.
\begin{definition}
\label{EC definition}
$\EC$ 
is the identity operator 
\[
\EC \colon \ContFuncs(\Nats,\Sier) \to \ContFuncs(\Nats,\ZO)
\]
that maps
a function
$\ContFuncs(\Naturals, \Sier)$
to itself in 
$\ContFuncs(\Naturals, \ZO)$.
\end{definition}
The following result is folklore.
\begin{lemma}
	\label{EC-is-a-cylinder}
	The operator $\EC$ is a cylinder.
\end{lemma}
\begin{proof}
By \cite[Lem.~6.3]{MR2791341}, the operator $\EC$ is strongly equivalent to the parallelization of LPO, which by \cite[Prop.~6.5]{MR2791341} is a cylinder.
\end{proof}

Note that $\EC$ applied to the enumeration of some set yields an output that is computable from the halting problem relative to that enumeration. Furthermore, every Turing degree contains an enumeration such that applying $\EC$ yields a set in the Turing jump of that degree.

\subsubsection{The $\Lim$ operator}
\label{limopsection}

Limit operators are fundamental in the program of calibrating the difficulty of analytic problems. For several problems within analysis that are equivalent to a limit operator, see \cite[\S2.3]{PaulyFouche}.

Furthermore, the limit operator on a sufficiently complicated space is equivalent to $\EC$ (see Lemma~\ref{ECeqWLim}), and hence plays a role among the Weihrauch degrees
somewhat analogous to that of the halting set within Turing computability.

\begin{definition}
\label{Lim definition}
Let $S$ be a metric space, and define
$\Lim_{S} : \subseteq \ContFuncs(\Nats,S) \to S$ to be the partial map 
taking Cauchy sequences to their limit points.
Write $\Lim$ to denote $\Lim_{\Cantor}$.
\end{definition}
	
We will make use of the following result in \cref{specific}.
\begin{lemma}[{\cite[Prop.~9.1]{MR2099383}}] 
	\label{compembeds}
	Suppose that $\Cantor$ computably embeds into a computable metric space $S$. Then $\Lim_S \eqsW \Lim$.
\end{lemma}
In both Sections~\ref{continuous} and \ref{specific}, we will use the strong Weihrauch equivalence of $\EC$ and $\Lim$.

\begin{lemma}[{\cite[Fact~3.5]{Brattka2012623}}]
\label{Lim-is-a-cylinder}
	The operator $\Lim$ is a cylinder.
\end{lemma}

\begin{lemma}
\label{ECeqWLim}
$\EC \eqsW \Lim$.
\end{lemma}
\begin{proof}
	By \cite[Prop.~9.1]{MR2099383}, we have 
	$\EC \eqW \Lim$.
(For how this aligns with the definitions above, see 
\cite[Def.~7.7]{MR2099383} in the case $k=1$.)
But $\EC$ and $\Lim$ are both cylinders by Lemmas~\ref{EC-is-a-cylinder} and \ref{Lim-is-a-cylinder}, and so by Lemma~\ref{cylinder-W-sW-lemma} we have $\EC \eqsW \Lim$.
\end{proof}

There are two key ingredients in the proof of the main theorem of \cref{continuous}:
$\EC \lesW \DD_{\Nats, [0,1]}$
and $\DD_{S,T} \lesW \EC$ for arbitrary
computable metric spaces $S$ and $T$. In \cref{specific}, we make use of reductions involving both $\EC$ and $\Lim_{\ProbMeasures(S)}$ for a particular space $S$.

\section{The disintegration operator on distributions admitting a unique continuous disintegration.}
\label{continuous}

We now present lower bounds (\cref{lower-uniform}) and upper bounds (\cref{upper-uniform}) for the conditioning operator in the context of measures admitting unique continuous disintegrations.  
The latter makes use of key results in \cref{Tjursection} pertaining to Tjur points.

\subsection{Lower Bound}
\label{lower-uniform}
The following proposition amounts to a relativized version of the main results of \cite{AFR10},
which demonstrates that
the conditional distribution of a computable random variable given another
need not be computable.
We thank an anonymous referee for suggesting a 
construction that greatly simplified the following
proof.

\begin{proposition}
\label{ECleD}
$\EC \lesW \DD_{\Nats, [0,1]}$.
\end{proposition}

\begin{proof}
Let $x \in \ContFuncs(\Nats,\Sier)$. 
From $x$, we can compute a function $y \in \ContFuncs(\Nats \times \Nats, \ZO)$ such that $y(m,\cdot)$ is nondecreasing and, for every $m \in \Nats$, we have $x(m) = 1$ if and only if there exists an $n \in \Nats$ such that $y(m,n)=1$.  

Define the function $\iota : \Nats \to \Nats \cup \{\infty\}$ by $\iota(m) = \inf \{ k < \infty \st y(m,k) = 1 \}$,
with the convention that $\inf \, \emptyset = \infty$.

Let $f_k,f_\infty : [0,1] \to [0,2]$, for $k \in \Nats$, be defined by $f_\infty(z) = 1$ and $f_k (z) = 1 + \cos ( 2^{k+1} \pi z)$. 
\begin{claim}\label{indvcomp}
The measure $\mu_k$ with density $f_{\iota(k)}$ with respect to Lebesgue measure on $[0,1]$ is 
computable from $y$, uniformly in $k$.
\end{claim}
\begin{claimproof}
By \cref{measure-computable from basis}, 
it suffices to prove that $\mu_k$ is computable from $y$ on a Lebesgue-almost-decidable basis which is computably closed under finite unions\, uniformly in $k$ and for which the union map is computable.
Let $ i,j,m \in \Nats$ with $i < j \le 2^m$,
let
\[
J_{i,j,m} \defas \Bigl ( \frac {i} {2^m}, \frac {j } {2^m} \Bigr )
\]
be an open interval with dyadic end points, 
and let $\mathcal J$ be the collection 
of all finite disjoint unions of such sets, indexed in the obvious way.  
Then $\mathcal J$ is a basis of continuity sets
 with respect to any measure absolutely continuous with respect to Lebesgue measure.  
It is straightforward to show that $\mathcal J$ is moreover a Lebesgue-almost-decidable basis.
Uniformly in $m \in \Nats$, the functions $f_m,f_\infty$ are computable and so their definite integrals over $\mathcal J$ are uniformly computable.
It follows that $\mu_k \bigl(J_{i,j,m}\bigr)$ is computable, uniformly in $i,j,k,m$, for every $m$ in the set
$\{m \in \Nats : m \ge \iota(k) \}$, which is itself a computable element of $\Opens(\Nats)$ from $y$, uniformly in $k$.
It suffices to show that
$\mu_k \bigl( J_{i,j,m} \bigr)$ is computable, uniformly in $i,j,k,m$ for 
every $m$ in the set
$\{m \in \Nats : m < \iota(k) \}$, which is also a computable element of $\Opens(\Nats)$ from $y$, uniformly in $k$.
To see this, note that, for $ m < \iota(k)$, 
\[
	\int_{J_{i,j,m}} f_{\iota(k)}(z) \,\dee z = 2^{-m} (j-i),
\] 
completing the proof of the claim.
\end{claimproof}	

We now define $\mu$ as the measure on $\Nats \times [0,1]$ such that $\mu( \{n\} \times \cdot ) $ is the measure having density $g_n$ with respect to Lebesgue measure, where
\[\label{gdefn}
g_{2m} \defas 2^{-m-2} f_{\iota(m)}
\quad\text{ and }\quad
g_{2m+1} \defas 2^{-m-2}(2- f_{\iota(m)}). 
\]
By construction $\sum_{n \in \Nats} g_n(x) = 1$ for every $x \in [0,1]$, and so $\mu$ is a probability measure with $\supp(\PF {\pi_2} \mu) = [0,1]$.
Note that these measures are easily seen to be uniformly computable in $y$ from \cref{indvcomp}, 
and so, again by \cref{measure-computable from basis}, 
it follows that $\mu$ is computable.

It follows from an application of Bayes' rule that the disintegration of $\mu$ along its second coordinate
agrees almost everywhere with the map $\kappa : [0,1] \to \ProbMeasures(\Nats)$ given by
\[\label{g is comp}
\kappa(z)\{n\} = g_n(z).
\]
Clearly, $\kappa$ is continuous, and because $\PF {\pi_2} \mu$ has full support, $\kappa$ is the unique continuous disintegration.

In summary, we have described a computable map $\ContFuncs(\Nats,\Sier) \to \ACD_{\Nats, [0,1]}$ sending a function $x$ to a measure $\mu$ admitting a unique continuous disintegration.
Take $K$ to be 
a
realizer for this map.
We now show how to compute $x$ as an element in $\ContFuncs(\Nats,\ZO)$ from $\kappa$, which gives a realizer $H$ such that
$ H \circ G \circ K \vdash \EC$ for any $G \vdash \ACD_{\Nats,[0,1]}$.

Note that, by \cref{g is comp}, the function $g_k$ is computable from $\kappa$, uniformly in $k$.  It follows that the functions $f_{\iota(k)}$ are computable from $\kappa$, uniformly in $k$.  Observe that
\[
	x(k) = f_{\iota(k)}(0) - 1 = 
\begin{cases}
	1, & \textrm{if~}\, \iota(k) < \infty; \\
	0, & \textrm{if~}\, \iota(k) = \infty.
\end{cases}
\]
It follows that $x(k)$ is computable from $\kappa$, uniformly in $k$, as desired.
\end{proof}

Recall that $\DD$ denotes the operator $\DD_{[0,1],[0,1]}$.
We thank an anonymous referee for suggestions that simplified the following proof.

\begin{corollary}\label{ECDD}
$\DD_{\Nats, [0,1]} \lesW \DD$ and so $\EC \lesW \DD$.
\end{corollary}
\begin{proof}
Let $\alpha \colon \Nats \to [0,1]$ be some canonical computable injective map. In particular, assume there exists a computable map $\phi \colon \Nats \to \Opens([0,1])$ such that $\alpha^{-1}[ \phi_n ] = \{n\}$.
Given $\mu \in \ProbMeasures(\Nats \times [0,1])$,
define $\nu \in \ProbMeasures([0,1]^2)$ to be the pushforward of $\mu$ via $\alpha \times \id$.
Then
$
\DD_{\Nats, [0,1]}(\mu)(t)\{n\}
=
\DD(\nu)(t)(\phi_n),
$ 
the latter being an element of $\LReals$, computable from 
$\DD_{\Nats, [0,1]}(\mu)$ and $t$,
uniformly in $n$.  
This demonstrates $\DD_{\Nats, [0,1]} \lesW \DD$.
Then $\EC \lesW \DD$ follows from \cref{ECleD}.
\end{proof}

\subsection{Upper Bound}
\label{upper-uniform}
We now make use of results about Tjur points from \cref{Tjurpoints}.
Recall that for a computable metric space $S$ and $\mu \in \ProbMeasures(S)$, 
we have defined $\AD {\mu}(S)$ to be the collection of 
$\mu$-continuity sets in $\Opens(S)$.

\begin{lemma}\label{conditioningoncontset}
For any computable metric space $S$,
the function 
taking $\mu \in \ProbMeasures(S)$ and $H \in \AD {\mu}(S)$, satisfying $\mu(H) > 0$, to
the probability measure 
$\mu^H$
is computable.
\end{lemma}
\begin{proof}
Let $A \in \Opens (S)$. Because $H \in \Opens(S)$, we have $A \cap H \in \Opens(S)$ and so $\mu(A \cap H) \in \LReals$.
But, by \cref{compmeasureofad}, $\mu(H) \in \Reals$ is computable from $H$, 
and so $\mu^H(A) = \frac { \mu( A \cap H) }{ \mu(H) }$,
viewed as an element of $\LReals$, is 
is computable from $H$ and $A$, because division of an element of $\LReals$ by an element of $\Reals \setminus \{0\}$ is computable. 
\end{proof}

Recall that $\pi_2 : S \times T \to T$ denotes the projection map.

\begin{corollary}\label{compcond}
Let $S$ and $T$ be computable metric spaces,
let $\mu \in \ProbMeasures(S \times T)$,
let $\mu_T$ be its projection onto $T$,
and let $U \in \AD {\mu_T}(T)$ be a $\mu_T$-continuity set with $\mu_T(U)>0$.
Then $\mu_{\pi_2}^U = \mu^{S \times U}$
is computable from $\mu$ and $U$.
\end{corollary}
\begin{proof}
We have $\mu(S \times U) = \mu_T(U) > 0$.
Moreover, $\mu(\partial (S \times U)) = \mu_T(\partial U) = 0$, 
and so 
$S \times U  \in \AD {\mu}(S \times T)$.
It follows that $\mu^{S \times U}$ is computable by \cref{conditioningoncontset}.
\end{proof}

\begin{proposition}
\label{DleLim}
For any two computable metric spaces $S$ and $T$, we have
\[
  \DD_{S,T} \lesW \EC.
\]
\end{proposition}

\begin{proof}
	By Lemma~\ref{EC-is-a-cylinder}, the operator $\EC$ is a cylinder. Hence by
	Lemma~\ref{cylinder-W-sW-lemma},
	it suffices to show
		that
	$\DD_{S,T} \leW \EC$.
	Let $\mu \in \ACD_{S, T}$,
let $\mu_T$ be its projection onto $T$,
and let $\pmetric$ denote the (computable) metric on $\ProbMeasures(S)$.

By Lemma~\ref{continuitysubbasis}, we can compute a $\mu$-continuous basis $B$ from $\mu$.
By \cref{compcond},
$\mu^{B(n)}_{\pi_2} \in \ProbMeasures(S)$ is computable from $\mu$, uniformly in $n$, and so, the set $F$ of those triples $(m,n,k)$ such that 
$\pmetric(\mu^{B(m)}_{\pi_2},\mu^{B(n)}_{\pi_2}) > 2^{-k}$ is also computably enumerable from $\mu$. 
Define $\xi \in \ContFuncs(\Nats \times \Nats,\Sier)$ by
\[
\xi(n,k) = 
\begin{cases}
1, &\text{if $\exists m \in \Nats$, $m \ge n$ and $(m,n,k) \in F$, }\\
0, &\text{otherwise.}
\end{cases}
\]
As an element of $\ContFuncs(\Nats \times \Nats,\Sier)$,
note that $\xi$ is computable from $I$ and $F$.

Let $K$ be the map taking $\mu$ to $\xi$, and 
define $\hat \xi \in \ContFuncs(\Nats \times \Nats, \ZO)$ by $\hat \xi \defas (\EC \circ K)(\mu)$.
Because $\mu \in \ACD_{S, T}$, by Proposition~\ref{Tjur = continuous disintegration on measure one set} every point of $T$ is a Tjur point for $\mu$ along $\pi_2$. 
Let $t \in T$.
As $t$ is a Tjur point, for every $k \in \Nats$, there is some $n_k \in \Nats$ such that $\hat\xi(n_k,k) = 0$ and $t \in B(n_k)$. Further, for such an $n_k$ we have $\pmetric(\mu^{B(n_k)}_{\pi_2}, \mu^t) \leq 2 \cdot 2^{-k}$. 

Define a sequence $\nu \in \ContFuncs(\Nats,\ProbMeasures(S))$ by $\nu(k) = \mu^{B(n_k)}_{\pi_2}$.
By construction of $\hat\xi$, the sequence $\nu$ is a rapidly converging Cauchy sequence converging to $\mu^t$ and so $\mu^t$ is computable from $\nu$ by 
the definition of a computable metric space.

In summary, there is a computable map taking the tuple $((\mu,\hat \xi),t)$ to $\mu^t$.
Let $H$ be the corresponding ``curried'' map 
taking the pair $(\mu, \hat \xi)$ to the map $t \mapsto \mu^t$, which is also computable.
Then $H$
satisfies $\DD_{S,T} = H \circ \langle \id, \EC \circ K \rangle$.
\end{proof}

\subsection{Equivalence}

The following result is an immediate consequence of 
\cref{ECleD,DleLim,ECDD}, as well as the fact that 
$\Lim \eqsW \EC$ by Lemma~\ref{ECeqWLim}.

\begin{theorem}\label{continuous degree equivalence}
$\DD \eqsW \DD_{\Nats, [0,1]} \eqsW \Lim$. \qed
\end{theorem}


\section{The disintegration of specific distributions}
\label{specific}

In the previous section, we considered the disintegration operator on the 
space of measures for which the disintegration operator is continuous (and hence computable from a single real). In this section, we consider the possible Weihrauch degrees
of disintegrations which may not be continuous.
We will show that 
there exist individual disintegrations that are themselves
strongly Weihrauch equivalent to $\Lim$.

\subsection{Definitions}

In order to study disintegration in the context of
probability measures that do not necessarily admit continuous disintegrations,
we work with the notions developed by Fraser and Naderi based on the Vitali covering property, as described in \cref{weakerFN}.

\begin{definition}
Let $T$ be a computable metric space and let $\mu$ be a probability measure on $T$.
A class $V \subseteq \Borel {T}$ of $\mu$-continuity sets has the 
\defn{strong Vitali covering property} with respect to $\mu$ when (1) it has the Vitali covering property with respect to $\mu$
and (2)
there exists a multi-valued map $U : T \multito \ContFuncs(\Nats,\AD {\mu}(T))$, computable relative to $\mu$, 
such that, for every $x \in T$, there is a sequence
$\< E_i \>_{i \in \Nats}$ converging regularly to $x$ with respect to $\mu$ such that
\[
	U(x)(n) \subseteq E_n
	  \text{ and }
	\mu(U(x)(n)) = \mu(E_n),
\]
for every $n \in \Nats$.
\end{definition}

The following is an important example of a collection of sets with the computable Vitali covering property.

\begin{lemma}
\label{computable Vitali for ultrametric}
If $T$ is a computable ultrametric space then the collection of open balls has the computable Vitali covering property with respect to any measure on $T$.
\end{lemma}
\begin{proof}
By \cref{ultrametric-has-Fraser} we know that the collection of open balls has the Vitali covering property. However if $T$ is a computable metric space then we can computably find a map 
that takes each element $x$ in $T$ to a sequence of open balls that converges regularly to $x$.
\end{proof}

Similarly, for $\Reals^n$, the collection of closed sets also has the computable Vitali covering property with respect to Lebesgue measure.

\begin{definition}
\label{Definition: Conditional Distribution for non Tjur}
Let $S$ and $T$ be computable metric spaces, 
let $\mu \in \M_1(S\times T)$,
let $\mu_T$ be its projection onto $T$,
and let $V \subseteq \Borel {T}$ be a class of $\mu_T$-continuity sets having the computable Vitali covering property with respect to $\mu_T$.
Define $\Dstar{S, T}(\mu,V)$ to be the map sending $t\in T$ to 
$\iota_1 ( \mu^t_{\pi_2,V} )$,
where $\mu^t_{\pi_2,V}$ is as defined in \cref{Fraser-disintegration thm}
and $\iota_1$ is as defined after \cref{ACD}.
\end{definition}

\begin{remark}
Observe that, by \cref{Fraser-disintegration thm}, $\Dstar{S,T}(\mu,V)$ is equal to the composition of the map $\iota_1$ with a disintegration of $\mu$ along $\pi_2$.
\end{remark}

\begin{remark}
By \cref{maintjurlemma}, 
we see that, for every $\mu \in \ACD_{S,T}$ 
and class $V \subseteq \Borel {T}$ of $\PF {\pi_2} \mu$-continuity sets having the computable Vitali covering property,
 $\Dstar{S,T}(\mu,V)=\DD_{S,T}(\mu)$.
\end{remark}

\subsection{Upper Bound}
\label{upper-worstcase}

\begin{proposition}\label{disintthm}
Let $S$ and $T$ be computable metric spaces, let
 $\mu\in \ProbMeasures(S\times T)$ be computable, 
 let $\mu_T$ be its projection onto $T$,
 let $V \subseteq \Borel {T}$ be a class of $\mu_T$-continuity sets with the strong Vitali covering property,
 and assume that $\Dstar{S,T}(\mu,V)$
 is total.
Then there exists a 
computable map $K \colon T \to \ContFuncs(\Nats,\ProbMeasures(S))$ 
such that
\[
\Dstar{S, T}(\mu,V) = \Lim_{\ProbMeasures(S)} \circ K.
\]
In particular,
\[
\Dstar{S, T}(\mu,V) \lesW \Lim_{\ProbMeasures(S)}.
\]
\end{proposition}
\begin{proof}
Note that, if $\mu_T(H) = \mu_T(H') > 0$ and $H' \subseteq H$, then $\mu^H_{\pi_2} = \mu^{H'}_{\pi_2}$.
Now, let $t \in T$.  
By the strong Vitali covering property, we may compute from $t$ 
an element $U \in \ContFuncs(\Nats,\AD {\mu}(T))$ 
such that 
$\nu(n) \defas \mu^{E_n}_{\pi_2} = \mu^{U_n}_{\pi_2}$
for all $n\in\Nats$, where 
$\<E_i\>_{i\in\Nats}$  is some sequence of elements of $V$
converging regularly to $t$ with respect to $\mu_T$.
By \cref{compcond}, $\nu(n)$ is computable from $t$, uniformly in $n$.
By Definition~\ref{Definition: Conditional Distribution for non Tjur}, the limit of $\nu$ exists and is $\Dstar{S, T}(\mu,V)(t)$.
Take $K$ to be the map sending $t$ to $\nu$.
\end{proof}

\subsection{Lower Bound}
\label{lower-worstcase-referee}

\newcommand{\ZZ}{\mathcal {C}}
\newcommand{\VZZ}{V_{\ZZ}}

Let $\II \defas [0,1] \setminus \{\frac{a+1}{2^b}\st a, b \in\Nats\}$. There is then a computable injection $i\colon \II \to \Cantor$ which takes each element to its binary representation. Further $\II$ is a Lebesgue measure one subset of $[0,1]$ and Lebesgue measure on $\Cantor$ is the pushforward of Lebesgue measure on $\II$ along $i$.
Let $\ZZ = \Cantor \times \Cantor$, which is an computable ultrametric space (i.e. a computable metric space where the metric is an ultrametric). Then $\Nats \times \ZZ$ can be viewed as a computable ultrametric space as well. Let $\VZZ$ be the collection of open balls in $\ZZ$. Note that, by \cref{ultrametric-has-Fraser}, the class $\VZZ$ has the Vitali covering property with respect to any measure. Because every open ball in an ultrametric space is a continuity sets with respect to every measure, it is also immediate that $\VZZ$ has the strong Vitali covering property with respect to every measure.

Write $a_n \sim b_n$ when $\frac {a_n}{b_n} \to 1$ as $n \to \infty$.
The proof of the next proposition follows a suggestion from an anonymous referee and is significantly simpler than the original one presented.

\begin{proposition}
\label{continuous lower bound-referee version}
There is a computable distribution $\mu \in \ProbMeasures(\Nats \times \ZZ)$ such that 
its disintegration $\Dstar{\Nats, \ZZ}(\mu,\VZZ)$ 
(along the projection $\pi: \Nats \times \ZZ \to \ZZ$) 
exists everywhere and satisfies $\EC \lesW \Dstar{\Nats, \ZZ}(\mu,\VZZ)$.
\end{proposition}

\begin{proof}

In \cref{ECleD}, each $x \in \ContFuncs(\Nats,\Sier)$ is mapped to a measure $\mu_x$ over $\Nats \times [0,1]$ admitting a total continuous disintegration along the projection $\Nats \times [0,1] \to [0,1]$. Further, for each $x \in \ContFuncs(\Nats,\Sier)$ we have $\mu_x(\II) = 1$. This implies that the map $\mu_x \to \mu_x^{\II}$ is computable and also that for each $x \in \ContFuncs(\Nats,\Sier)$, the measure $\mu_x^{\II}$ admits a total continuous disintegration along the projection $\Nats \times \II \to \II$ with $\mu_x^y = (\mu_x^\II)^y$ for all $y \in \II$. 

Let $\eta_x = \PF i {\mu_x^{\II}}$, be the pushforward of $\mu_x^{\II}$ along $i$. Note there is a map $\alpha$ that takes an open ball $B$ in $\Cantor$ to an open ball $\alpha(B) \subseteq i^{-1}(B) \subseteq \II$ such that for all $x \in \Cantor$, $\nu_x(B) = \mu_x^{\II}(\alpha(B))$. As such, the map $\mu_x^{\II} \mapsto \eta_x$ is computable and so the map $x \mapsto \eta_x$ is also computable. Further it is immediate that for every $y \in \II$, $i(y)$ is a Tjur point (for the projection map onto $\Cantor$) and $\eta_x^{i(y)} = \mu_x^{y}$.

For each $x$, let $\nu_x \in \ProbMeasures(\Nats)$ be $\eta_x^{i(0)}$ ($= \mu_x^0$), i.e., the image of $i(0)$ under the continuous disintegration of $\eta_x$ along the projection. 
From \cref{gdefn,g is comp}, we see that, for every $x$ and $k \in \Nats$,
\[
x(k) = 2^{k+2} \, \nu_x\{2k\}.
\]
Hence $x$ is a computable element of $\ContFuncs(\Nats,\ZO)$ from $\nu_x$.
Let $H$ be a realizer for the map from $\ProbMeasures(\Nats)$  to $\ContFuncs(\Nats,\ZO)$ taking $\nu_x$ to $x$.

Consider the map $\rho$ from $\Cantor$ to $\ContFuncs(\Nats,\Sier)$ such that
$\rho(s) = x$ if and only if, for all $ k \in \Nats$, 
\[
   x(k) = 1 
   \iff \text{$\exists n \in \Nats$ s.t.\ the string } 0(1^k0)^n \text{ appears in $s$ at position $n$. }
\]
Call an element $s \in \Cantor$ \emph{$\rho$-faithful} when, for all $k\in \Nats$, it holds that $s$ contains the substring $01^k0$ if and only if $\rho(s)(k) = 1$.
It is straightforward to establish that $\rho$ is a total computable surjective map 
admitting a computable multi-valued inverse 
realized by a map $\bar\rho$ such that $\bar\rho(x)$ is $\rho$-faithful for all $x$. 
Now endow $\Cantor$ with the uniform measure $\lambda$ and
define $\mu \in \ProbMeasures ( \Nats \times \Cantor \times \Cantor )$ as the mixture given by
\[
\mu(A \times B ) = \int_B \eta_{\rho(s)} (A)\, \lambda(\dee s),
\]
for all measurable subsets $A \subseteq \Nats \times \Cantor$ and $B \subseteq \Cantor$,
and let $\kappa=\Dstar{\Nats, \ZZ}(\mu,\VZZ)$ be the disintegration of $\mu$
along the projection $\pi:\Nats \times \ZZ \to \ZZ$.
Let $G$ be a realizer for $\kappa$.

For a finite string $s'$,
write $[s']$ for the clopen subset of $\Cantor$ comprised of all strings with prefix $s'$.
For $n \ge 0$ and $s \in \Cantor$, 
let $s_n$ be the length-$n$ prefix of $s$,
let $B^s_n = [0^n] \times [s_n]$, 
and note that $B^s_n \in \VZZ$ 
and that $B^s_1,B^s_2,\dotsc$ converges regularly to $(0^\w,s)$ with respect to Lebesgue measure. 
It follows that, for every $s \in \Cantor$ and subset $A \subseteq \Nats$,
\[
\kappa((0^\w,s))(A) = \lim_{n \to \infty} \frac { \mu( A \times B^s_n ) } {\mu ( \Nats \times B^s_n ) },
\]
provided the limit exists.

\begin{claim}\label{mainclaimfraser}
For every $\rho$-faithful $s \in \Cantor$,
we have $\lim_{n \to \infty} \frac { \mu( A \times B^s_n ) } {\mu ( \Nats \times B^s_n ) } = \nu_{\rho(s)}(A)$.
\end{claim}
\begin{claimproof}
Let $s \in \Cantor$ be $\rho$-faithful, let $x \in \ContFuncs(\Nats,\Sier)$ be the image of $s$ under $\rho$,
and let $k \in \Nats$.
For every $n \in \Nats$, let
$
E_n = \{ s' \in [s_n] \st \rho(s')(k) = x(k) \}
$
and note that
\[
\mu( \{k\} \times [0^n] \times E_n) 
&= \eta_x(\{k\} \times [0^n]).
\]
If $x(k) = 1$ then there exists $n_0 \in \Nats$, 
such that $E_n = [s_n]$ for every $n \ge n_0$.
If $x(k) = 0$, then, for $ n \ge k+1$, 
\[
\frac {\lambda (E_n)}{\lambda ([s_n])} 
\ge 1 - \sum_{i=0}^\infty 2^{- (k+1)(n+i) - 1} 
\ge 1 - \sum_{i=0}^\infty 2^{- n - i} 
= 1 - 2^{-n+1},
\]
where
the first inequality follows from the $\rho$-faithfulness of $s$ and a union bound on the event that $0(1^k0)^{n+i}$ appears at position $n+i$ for some $i \in \Nats$.
Therefore, in general,
\[\label{rhoreq}
\lambda (E_n) 
\sim \lambda ([s_n]), \text{ for } n \to \infty,
\]
and so
\[
\mu( \{k\} \times [0^n] \times [s_n])
\sim \mu( \{k\} \times [0^n] \times E_n), \text{ for } n \to \infty.
\]
But then
\[
\lim_{n \to \infty} 
   \frac {\mu (\{k\} \times B_n^s)}{\mu(\Nats \times B_n^s)}
 = \lim_{n \to \infty} 
   \frac {\eta_x (\{k\} \times [0^n])}{\eta_x(\Nats \times [0^n])} 
 = \nu_x (\{k\}),
\]
where the final equality follows from
the fact, established in \cref{ECleD}, 
that $0^\w$ is a Tjur point of $\eta_x$ (along the projection onto $\Cantor$),
completing the proof of the claim.
\end{claimproof}

Now let $x \in \ContFuncs(\Nats,\Sier)$. 
From $x$, we can, by assumption, 
compute a $\rho$-faithful element $s \in \Cantor$ such that $\rho(s) = x$.
Let $K$ be a realizer for the map taking $x$ to $(0^\w,s) \in \ZZ$.
By the above claim, 
the image of $(0^\w,s)$ under $\kappa$ 
is $\nu_x$, and so the composition $G \circ K$ is a realizer for the map taking $x$ to $\nu_x$.  Finally, the map realized by $H$ above takes $\nu_x$ to $x$ as an element of $\ContFuncs(\Nats,\ZO)$, 
completing the proof.
\end{proof}

\begin{remark}
The particular choice of $\rho$ is essential here: different surjections from $\Cantor$ to $\ContFuncs(\Nats,\Sier)$ do not necessarily lead to a reduction.
For example, consider the  map $\rho'$ taking a binary sequence to the set of all 
natural numbers $k$ such that $01^k0$ appears in the sequence. 
Like $\rho$, the map $\rho'$ is
a total computable surjective map admitting a computable multi-valued inverse whose images are $\rho'$-faithful. However, \cref{rhoreq} fails to hold for $\rho'$.  
The source of the failure is the fact that, for $\lambda$-almost all $s \in \Cantor$, we have $\rho'(s) = c_1$ where $c_1(k) = 1$ for all $k \in \Nats$!  
Had we used $\rho'$, then the disintegration of $\mu$ obtained by taking limits would have been the constant map $(u,s) \mapsto \nu_{x'}$.
It is surprising, though easy to verify, that the disintegration $(u,s) \mapsto \nu_{\rho'(s)}$ is a version of this disintegration, but it does not arise from the limiting construction of the disintegration when $\mu$ is constructed from $\rho'$.
\end{remark}

\begin{theorem}
There are computable ultrametric spaces $S$ and $T$ and a computable distribution $\mu \in \ProbMeasures(S \times T)$ such that if $V$ is the collection of open balls in $T$ then
the disintegration $\Dstar{S, T}(\mu,V)$ 
(along the projection $\pi: S \times T \to T$) 
exists everywhere and satisfies $\Lim \eqsW \Dstar{S, T}(\mu,V)$.
\end{theorem}
\begin{proof}
Let $S = \Nats$, $T = \ZZ$, and $V = \VZZ$ with $\ZZ, \VZZ$, and $\mu$ as in \cref{continuous lower bound-referee version}. Then $\Lim \lesW \Dstar{S, T}(\mu,V)$ follows immediately from \cref{continuous lower bound-referee version} and \cref{ECeqWLim}.
\nobreak Finally, $\Dstar{S, T}(\mu,V) \lesW\Lim $ follows immediately from \cref{disintthm} and \linebreak \cref{compembeds}. 
\end{proof}

\section*{Acknowledgments}
The authors would like to thank Arno Pauly for helpful conversations
regarding Weihrauch reducibility (and in particular the argument for \cref{EC-is-a-cylinder}),
and the anonymous referees for detailed comments on a draft, including 
suggestions that greatly simplified the proofs of \cref{ECleD,ECDD,continuous lower bound-referee version}.
The authors would also like to thank
Persi Diaconis for pointing us to Tjur's work, 
Laurent Bienvenu for useful discussions,
and Quinn Culver, Bj{\o}rn Kjos-Hanssen, and Geoff Patterson
for comments on an earlier version of this material.

Work on
this publication was made possible through the support of 
ARO grant W911NF-13-1-0212,
ONR grant N00014-13-1-0333,
NSF grants DMS-0901020 and
DMS-0800198, 
DARPA Contract Award Numbers FA8750-14-C-0001 and FA8750-14-2-0004,
and grants from the
John Templeton Foundation and Google. The opinions expressed in this publication are those of the authors and do not necessarily reflect the views of the John
Templeton Foundation or the U.S.\ Government.

This paper was partially written while CEF and DMR were
participants in the
program \emph{Semantics and Syntax: A Legacy of Alan Turing} at the
Isaac Newton Institute for the Mathematical Sciences.
DMR was partially supported
by graduate fellowships from the National Science Foundation and MIT Lincoln
Laboratory, by a Newton International Fellowship, Emmanuel Research Fellowship,
 and by an NSERC Discovery Grant and Connaught Award.



\renewcommand*{\bibfont}{\small}
\defbibheading{bibliography}{\section*{References}\markboth{}{}}
\printbibliography

\vfill 


\end{document}